\documentclass[table,x11names]{article}

\usepackage{graphicx}
\usepackage[margin=1in]{geometry}
\usepackage{mathtools,amssymb,amsthm}
\usepackage{url}
\usepackage{verbatim}
\usepackage{tikz}
\usetikzlibrary{positioning,arrows.meta}

\usepackage{array}
\newcolumntype{P}[1]{>{\centering\arraybackslash}p{#1}}

\usepackage{hyperref}
\hypersetup{
    colorlinks = true,
    final=true,
    plainpages=false,
    pdfstartview=FitV,
    pdftoolbar=true,
    pdfmenubar=true,
    pdfencoding=auto, 
    psdextra, 
    bookmarksopen=true,
    bookmarksnumbered=true,
    breaklinks=true,
    linktocpage=true
}
\usepackage[nameinlink, capitalise, noabbrev]{cleveref}

\usepackage{csquotes}

\usepackage{caption,subcaption}
\usepackage{float}       
\usepackage{booktabs}

\usepackage{xparse}
\let\realItem\item 
\makeatletter
\NewDocumentCommand\myItem{ o }{%
   \IfNoValueTF{#1}%
      {\realItem}
      {\realItem[#1]\def\@currentlabel{#1}}
}
\makeatother

\usepackage[shortlabels]{enumitem}    
\setlist[enumerate]{
    before=\let\item\myItem,       
    label=\textnormal{(\arabic*)}, 
    widest=(3A)                    
}

\usepackage[color=yellow]{todonotes}
\usepackage{todonotes}
\makeatletter
\define@key{todonotes}{TR}[]{%
    \setkeys{todonotes}{author=TR,color=blue!40,size=scriptsize}}%
\makeatother 

\usepackage[natbib, 
maxcitenames=3, 
maxbibnames=10,  
sorting=nyvt,
giveninits,
url=false,
doi=false,
isbn=false,
sortcites,
defernumbers,
backref,
backend=biber
]{biblatex}

\usepackage{pgfplots}
\usepackage[]{tikz}
\pgfplotsset{compat=1.9}

\crefname{algocf}{Algorithm}{Algorithms}
\usepackage{algorithm}
\usepackage{algpseudocode}

\usepackage{authblk}

\newtheorem{proposition}{Proposition}
\newtheorem{theorem}{Theorem}
\newtheorem{lemma}{Lemma}

\newtheorem{corollary}{Corollary}
\theoremstyle{definition}

\newtheorem{assumption}{Assumption}
\theoremstyle{remark}
\newtheorem{remark}{Remark}
\crefname{assumption}{Assumption}{Assumptions}

\DeclareMathOperator{\supp}{supp}

\newcommand{\R}{\mathbb{R}}

\newcommand{\E}{\mathbf{E}}
\newcommand{\calE}{\mathcal{E}}

\newcommand{\sol}{x^\star}
\newcommand{\OV}{\overline{V}}
\newcommand{\OX}{\overline{X}}

\newcommand{\cP}{\mathcal{P}}

\newcommand{\veps}{\varepsilon}

\newcommand{\lin}{\textup{lin}}
\newcommand{\err}{\textup{Err}}

\definecolor{SkyBlue}{HTML}{46C5DD}

\definecolor{cocoabrown}{rgb}{0.82, 0.41, 0.12}
\definecolor{guppiegreen}{rgb}{0.0, 1.0, 0.5}
\definecolor{grullo}{rgb}{0.66, 0.6, 0.53}
\definecolor{indiagreen}{rgb}{0.07, 0.53, 0.03}
\definecolor{brilliantrose}{rgb}{1.0, 0.33, 0.64}
\definecolor{midnightblue}{rgb}{0.1, 0.1, 0.44}
\definecolor{pyellow}{RGB}{221,170,51}
\definecolor{capri}{rgb}{0.0, 0.75, 1.0}
\definecolor{mirror}{named}{brilliantrose}

\definecolor{cadmiumgreen}{rgb}{0.0, 0.42, 0.24}
\definecolor{hotpink}{rgb}{1.0, 0.41, 0.71}
\definecolor{cgreen}{RGB}{0, 180, 100}
\hypersetup{
    colorlinks=true,
    linkcolor=cadmiumgreen,
    citecolor=cgreen,
    urlcolor=blue,
    linktocpage
}

\numberwithin{equation}{section}

\newcommand{\dt}{\Delta t}

\newcommand{\proba}{\mathbf P}
\newcommand{\expect}{\mathbf{E}}

\title{Long-time Stability and Convergence \\
of Particle Swarm Optimization}
\date{\today}

\author[1]{Giacomo Borghi}
\author[2]{Hui Huang}
\author[3]{Dohyeon Kim}

\affil[1]{Department of Mathematics, School of Mathematical and Computer Sciences (MACS), Heriot-Watt University, Edinburgh, UK}
\affil[2]{School of Mathematics, Hunan University,  Changsha, China}
\affil[3]{Department of Computing and Mathematical Sciences, Caltech, Pasadena, USA}

\begin{document}

\clearpage

\maketitle

\begin{abstract}

Particle Swarm Optimization (PSO) is a global optimization algorithm defined by an interacting set of particles evolving over the search space. Heuristically motivated, its theoretical analysis remains limited due to the second-order, stochastic, and highly nonlinear nature of the dynamics. In this paper, we connect classical PSO stability analysis under the stagnation assumption with more recent mean-field methods, providing new quantitative estimates for the time-discrete algorithm.
We study in particular a regularized PSO model without memory, with non-degenerate noise by adding a noise floor to the original model. 
Studying such a surrogate model allows us to identify quantitative conditions under which the dynamics is stable and converges toward a small neighborhood of a global minimizer. We do so by first studying the Schur stability of the linearized dynamics, then analyzing the convergence properties of a nonlinear mean-field system via a Laplace principle, and finally establishing a quantitative error bound for the mean-field approximation of order $N^{-1/2}$. 

\end{abstract}

\medskip

\textbf{Keywords:} Particle  swarm optimization, Schur stability, particle systems, mean-field limit, global optimization

\bigskip


\section{Introduction}

The goal of this paper is to study a long-time asymptotics of a global optimization algorithm on the discrete time level directly. More specifically, we focus on Particle Swarm Optimization (PSO), which is more broadly, an interacting particle based method, or metaheuristics, for solving an optimization problem. These algorithms, employ a set of agents, or particles, to stochastically explore the optimization search space and eventually converge towards a minimum. They are particularly effective in low-to-moderate-dimensional black-box problems where the objective function is non-differentiable, or too expensive to compute its gradient, or non-convex. These settings appear in hyperparameter optimization during the training  of machine learning models \cite{bergstra2011Algorithms,archambeau2024Hyperparameter,zito2025Metaheuristics}, or Bayesian inverse problems. Some well-known examples include Genetic Algorithms \cite{holland1975}, Differential Evolution \cite{storn1997differential}, Ant Colony Optimization (ACO) \cite{Dorigo1996}, PSO \cite{kennedy_eberhart_1995}, and Consensus-based Optimization (CBO) \cite{pinnau2017consensus}.

As their names suggest, the optimization dynamics are inspired by natural phenomena, such as the flocking of birds in the case of PSO, and ants behavior for \cite{Dorigo1996}. To solve optimization problems of type
\begin{equation}\label{eq:problem}
    \sol \in \underset{x \in \R^d}{\mathrm{argmin}}\, \calE(x),
\end{equation}
where $\calE:\R^d\to\R$ is a continuous objective function, each algorithm will evolve a swarm of particles or agents according to certain dynamics. In the standard PSO algorithm, each \textit{particle}, $i = 1,\dots,N$ is described by a position vector $X^i_n\in \R^d$ and a velocity $V_n^i\in \R^d$. The particles are attracted both toward its own best previously visited position $p^i_n$ (\emph{personal best}, or \(pbest\)) and toward the global best position found by the swarm $g_n$ (\emph{global best}, or \(gbest\)) at step $n\geq 0$ \cite{kennedy_eberhart_1995}. Inertia is typically included in the velocity update to stabilize the dynamics \cite{shi1998modified}, leading to an update of the form 
\begin{equation}\label{eq:classical-pso}
\begin{dcases}
X_{n+1}^i &= X_n^i + V_{n+1}^i \\
V_{n+1}^i & = (1 - \gamma) V_n^i +  c_1 r^{1,i}_{n}\odot \bigl(p^{i}_{n}-X^{i}_{n}\bigr)
    + c_2 r^{2,i}_{n}\odot\bigl(g_{n}-X^{i}_{n}\bigr), 
\end{dcases}
\end{equation}
for $i = 1,\dots,N$. Above, $c_1, c_2>0$ are parameters,  \(r_{1,i}^{n},r_{2,i}^{n} \) are random vectors, and $\gamma\in (0,1)$ is a friction parameter which introduces an inertia $m  = 1-\gamma$. With $\odot$ we indicate the component-wise multiplication between vectors.

In recent years there has been an effort to provide a more rigorous mathematical footing to such heuristic strategies by looking at them as interacting particle systems to be studied under the lens of statistical physics. This has allowed not only to understand their convergence properties \cite{delmoral2001Asymptotic,borghi2025ga,huang_qiu_riedl_2023,carrillo2018analytical}, but also to draw connections between different strategies and suggest improvements \cite{grassi_pareschi_2021, carrillo2021anisotropic,chen2020consensus}.
Along this line of work, we study a regularized version of PSO which preserves the key interaction mechanism, while also being amenable to mean-field approximation. This is in line with the convergence analysis provided for Consensus-Based Optimization (CBO) \cite{pinnau2017consensus}, which share some algorithmic similarities. Nonetheless, CBO is of first-order dynamics, while \eqref{eq:classical-pso} is a second-order model.

\subsection{Literature review}

The PSO update \eqref{eq:classical-pso} is non-linear due to the presence of the particles' personal best positions $p_n^i$ and global best positions $g_n$, which not only depend on the entire particle system, but also on its history. Moreover, the stochasticity of the system introduces an additional layer of complexity.

To study the stability of the system, a first classical approach consists of making the so-called \textit{stagnation} assumption \cite{clerc2002explosion, trelea2003dynamic, jiang2007stochastic, poli2009mean}
\begin{equation} \label{eq:stagnation} \tag{S}
p^i_n, g_n \approx const \,.
\end{equation}
Under \eqref{eq:stagnation}, early analyses reduce PSO to a linear second-order difference equation with fixed attractors. In this deterministic setting, \cite{clerc2002explosion} studies stability through the eigenvalues of the associated transition matrix, while \cite{kadirkamanathan2006stability} develops a related Lyapunov-based analysis. Later works retain the stochastic coefficients and characterize stability in terms of first- and second-order moments: \cite{jiang2007stochastic} and \cite{poli2009mean} derive spectral-radius conditions for convergence of the expectation and variance of the stagnated dynamics. Subsequent contributions relax strict stagnation by allowing the best positions to be random or partially time-dependent while remaining analytically decoupled from the particles: \cite{fernandez2011stability} considers stochastic attractors and oscillatory behaviour, \cite{liu2015order} studies a weak-stagnation regime where personal bests may improve while the global best is fixed, and \cite{bonyadi2016stability,cleghorn2018particle} model the bests through fixed or time-varying random vectors. In this line of work, the central question is the stability of the linearized PSO system under \eqref{eq:stagnation}, and there is no analysis of the interplay between particles.

A complementary analysis, initiated in \cite{grassi_pareschi_2021} and further developed in \cite{huang_2021_note, huang_qiu_riedl_2023}, regularizes the PSO dynamics to derive a mean-field approximation of the interacting particle system. To do so, a crucial step consists of regularizing the global best by using exponential weights. For a probability distribution $\rho\in \mathcal{P}(\R^d)$, the regularized best is given by the weighted point
\begin{equation}\label{eq:consensuspoint}
    x^\alpha[\rho]
    :=
    \frac{\int x\, e^{-\alpha \calE(x)}\,\rho(dx)}
         {\int e^{-\alpha \calE(x)}\,\rho(dx)}
         \qquad \textup{with}\quad \alpha \gg 1 .
\end{equation}
If one considers the empirical distribution $\rho^N_n = (1/N)\sum_{i=1}^N \delta_{X_n^i}$, then, provided the minimizer of \(\mathcal E\) over the current particle locations is unique,
\[
x^\alpha[\rho^N_n] \longrightarrow g_n \qquad \textup{as} \quad \alpha \to \infty \,.
\]
Thus \(x^\alpha[\rho_n^N]\) regularizes the best current particle position for  $\alpha\gg1$. In PSO models with memory, which is studied in \cite{grassi_pareschi_2021,huang_2021_note, huang_qiu_riedl_2023}, the same construction may instead be applied to regularized personal-best variables in order to approximate the historical global best. We will consider a memory-less PSO dynamics for simplicity. Numerically, it has shown that getting rid of the personal best has a negligible effect the algorithmic performance \cite{doi:10.1142/9789811266140_0003}. The corresponding stochastic differential equation (SDE) in this case reads, for $i=1,\dots,N$,
\begin{equation}
\label{eq:tc-pso}
\begin{cases}
    dX_t^i &= V^i_t dt, \\
    mdV_t^i &= - \gamma V_t^i dt + \lambda(x^\alpha[\rho^N_t] - X_t^i) + \sigma(x^\alpha[\rho^N_t] - X_t^i)\odot dB_t^i \,.
\end{cases} 
\end{equation}
In the continuous-time description, the velocity update is split into two components: a deterministic one, which depends on a parameter $\lambda>0$, and a stochastic one, which depends on a diffusion parameter $\sigma>0$ and independent Brownian processes $(B_t^i)_{t\geq 0}$, $i = 1,\dots,N$.

The key aspect of this regularized PSO dynamics is that, in the many-particle limit $N\to \infty$, the particle system can be approximated by a single-particle process of McKean--Vlasov type. This is the so-called mean-field approximation of the system. This is closely related to Propagation of chaos \cite{MR1108185}, which is an asymptotic independence of interacting particles as $N \to \infty$. Specifically, the empirical measure $\rho^N_t$ is approximated by a deterministic distribution $\rho_t\in \mathcal{P}(\R^d)$,
\begin{equation}\label{eq:asm:mf} \tag{MF}
\rho^N_t \approx \rho_t \qquad \textup{for}
\quad N \gg 1\,,
\end{equation}
leading, in turn, to an approximation of the regularized best point as $x^\alpha[\rho_t^N]\approx x^\alpha[\rho_t]$. The resulting nonlinear particle system is given by 
\begin{equation}
\label{eq:tc-pso-mf}
\begin{cases}
    d\OX_t &= \OV_t dt, \\
    md\OV_t &= - \gamma \OV_t dt + \lambda(x^\alpha[\rho_t] - \OX_t) + \sigma(x^\alpha[\rho_t] - \OX_t)\odot dB_t, \\
    \rho_t &=\mathrm{Law}(\OX_t) \,.
\end{cases} 
\end{equation}
The rigorous derivation of the mean-field limit $N \to \infty$ in the more general setting with personal bests has been studied in \cite{huang_2021_note}. In \cite{huang_qiu_riedl_2023}, the authors study the convergence of such mean-field PSO models towards solutions to \eqref{eq:problem}. And more recently, such a uniform in time mean-field limit result for a PSO model without personal best was shown in \cite{ha2026uniformintimepropagationchaossecondorder}, which was an extension of the previous work \cite{gerber2026uniformintimepropagationchaosconsensusbased} on the first-order dynamics.

The analysis of the PSO particle system via the mean-field approximation \eqref{eq:asm:mf} derives from the large body of literature on the Consensus-Based Optimization (CBO) algorithm, see, to name a few, \cite{pinnau2017consensus,carrillo2018analytical,carrillo2021anisotropic,fornasier2024convergence}. Indeed, the particle system \eqref{eq:tc-pso} can be seen as a second-order version of CBO, which can be recovered as the zero-inertia limit \cite{cipriani_huang_qiu_2022, doi:10.1142/9789811266140_0003}. Therefore, many of the techniques developed in this context have been transferred to the analysis of PSO-type interactions.

\subsection{Our contribution}

The objective of this work is to bring together the stability analysis based on the stagnation assumption \eqref{eq:stagnation} with the convergence towards minima based on the mean-field approximation \eqref{eq:asm:mf}. This will allow us to provide quantitative estimates of convergence towards solutions to \eqref{eq:problem} at large times.

We will consider a second-order PSO dynamics with non-degenerate noise,
\begin{equation}\label{eq:2nd-cbo-td-particle}
\begin{dcases}
    X_{n+1}^i = X_n^i + V_{n+1}^i,\\
    V_{n+1}^i
    = (1-\gamma)V_n^i
      + \lambda \bigl(x^\alpha[\rho_n^N]-X_n^i\bigr)
      + \sigma \bigl(\sigma_0 + |x^\alpha[\rho_n^N]-X_n^i|\bigr)\xi_n^i .
\end{dcases}
\end{equation}
The floor parameter $\sigma_0>0$ guarantees a strictly positive level of exploration at every step.
To be closer to the standard PSO dynamics \eqref{eq:classical-pso}, the dynamics is discrete in time, as actual implementable algorithms are. The random variables $\xi_n^i$ are i.i.d. standard Gaussian vectors. We note that introducing a non-degenerate diffusion via a baseline noise level $\sigma_0$ has already been considered in the CBO literature (see, e.g., \cite{bungert2025polarized, bianchi2025consensus, borghi2025bgk}). This approach prevents the premature convergence of particles to local minimizers—a phenomenon particularly prevalent when the global minimizer lies outside the support of the initial particle distribution, as observed in \cite{huang2025self, huang2025faithful} and further discussed in \cite{fornasier2026consensus}.
Although in this setting the particles are not expected to concentrate at the global minimizer due to the non-degenerate noise, the final output $x^\alpha[\rho_{n_T}^N]$ of the algorithm approximates the target as $\alpha \to \infty$ by virtue of the Laplace principle.

The main steps of our analysis are:
\begin{itemize}
\item \textbf{Mean-field approximation.}
We derive precise conditions under which the PSO-type particle system \eqref{eq:2nd-cbo-td-particle} converges to the corresponding mean-field single-particle process,
\begin{equation}\label{eq:2nd-cbo-td-mf}
\begin{dcases}
    \OX_{n+1} = \OX_n + \OV_{n+1},\\
    \OV_{n+1}
    = (1-\gamma)\OV_n
      + \lambda \bigl(x^\alpha[\rho_n]-\OX_n\bigr)
      + \sigma \bigl(\sigma_0 + |x^\alpha[\rho_n]-\OX_n|\bigr)\xi_n,\\
    \rho_n = \mathrm{Law}(\OX_n).
\end{dcases}
\end{equation}
By a coupling method, we show that the $N$-particle system tracks the mean-field dynamics with rate $\mathcal{O}(N^{-1/2})$. The analysis makes use of estimates derived for the CBO particle system \cite{gerber2025mean, Hui_Huang_2025_cpaa}, and in particular of the stability of the consensus point \eqref{eq:consensuspoint}.

\item \textbf{Long-time quantitative Laplace principle via a bootstrap argument.}
By relying on a quantitative version of the Laplace principle \cite{huang_qiu_riedl_2023}, we control the difference between the consensus point $x^\alpha[\rho_n]$ and a global solution $\sol$ to \eqref{eq:problem}. A key challenge in estimating $|x^\alpha[\rho_n] - \sol|$ at every step $n \geq 0$ is to bound from below the mass that $\rho_n$ allocates around the minimizer, and this is where the non-degenerate noise assumption is used. Using a bootstrap argument, we are able to provide a uniform-in-time bound of the form
\[
|x^\alpha[\rho_n] - \sol|^2 \lesssim \veps \qquad \textup{for all}\quad n\in \mathbb{N},
\]
provided $\alpha>0$ is sufficiently large.

\item  \textbf{Mean-square stability via error decay.}

The above bound on the consensus point $x^\alpha[\rho_n]$ suggests that the mean-field dynamics evolves almost in a stagnation regime \eqref{eq:stagnation}, since $\sol = const$. This allows us to study the stability of the linearized dynamics
\begin{equation}\label{eq:lin}
\begin{dcases}
    X_{n+1}^{\lin}=X_n^{\lin}+V_{n+1}^{\lin},\\
    V_{n+1}^{\lin}
    =(1-\gamma)V_n^{\lin}
    +\lambda(\sol-X_n^{\lin})
    +\sigma\bigl(\sigma_0+|\sol-X_n^{\lin}|\bigr)\xi_n,
\end{dcases}
\end{equation}
to infer stability of the nonlinear mean-field particle system.
We do this by studying a Lyapunov functional for random state vectors $(X,V)$,
\begin{equation} \label{eq:err_functional}
     \err(X,V)
     =\E\Big[\,|X-x^\star|^2
    +C_H\Big|V-\frac{\lambda}{\gamma}(x^\star-X)\Big|^2
    +2\theta\Big\langle X-x^\star,\,V-\frac{\lambda}{\gamma}(x^\star-X)\Big\rangle\Big]\,.
\end{equation}
which depends on paramaters $C_H, \theta$. 
Using Stein's theorem for the deterministic transition matrix $A$, we prove that for parameters \((\lambda,\gamma)\) satisfying the Schur stability condition \(\varrho(A)<1\), there exists a positive-definite metric \(P= \begin{pmatrix}
    1 & \theta\\
    \theta & C_H
    \end{pmatrix}\succ0\,,\) and equivalently, the parameters $C_H, \theta$ that captures the dissipative structure of the linearized dynamics. Under a suitable small-noise assumption, transferring this geometry to the nonlinear system \eqref{eq:2nd-cbo-td-mf} yields a geometric decay estimate up to a residual forcing term driven by the consensus error \(|x^\alpha[\rho_n]-\sol|^2\) and the ambient noise floor \(\sigma_0^2\).
\end{itemize}

\begin{figure}[t]
    \centering
    \begin{tikzpicture}[
        node distance=3.2cm and 4.2cm,
        box/.style={draw, rounded corners, align=center, inner sep=6pt},
        arr/.style={->, thick}
    ]

    \node[box] (pso) {Particle system \eqref{eq:2nd-cbo-td-particle}};
    \node[box, right=of pso] (mf) {Mean-field dynamics \eqref{eq:2nd-cbo-td-mf}};
    \node[box, below=of mf] (lin) {Linearized proxy \eqref{eq:lin}};
    \node[box, below=of pso] (sol) {Global minimizer \(\sol\)};

    \draw[arr] (pso) -- node[above] {$N\to \infty$ (Section \ref{sec: mfl})} (mf);
    \draw[arr] (mf) -- node[right,text width=3.5cm] {Bootstrap Laplace principle
    (Section \ref{sec: convergence})} (lin);
    \draw[arr] (lin) -- node[below,text width=4cm] {Lyapunov contraction (Section \ref{sec: contractivity})} (sol);
    \draw[arr, dashed] (pso) -- node[left] {long-time analysis} (sol);

    \end{tikzpicture}
    \caption{Paper structure and proof strategy blueprint.}
    \label{fig:proof_strategy}
\end{figure}
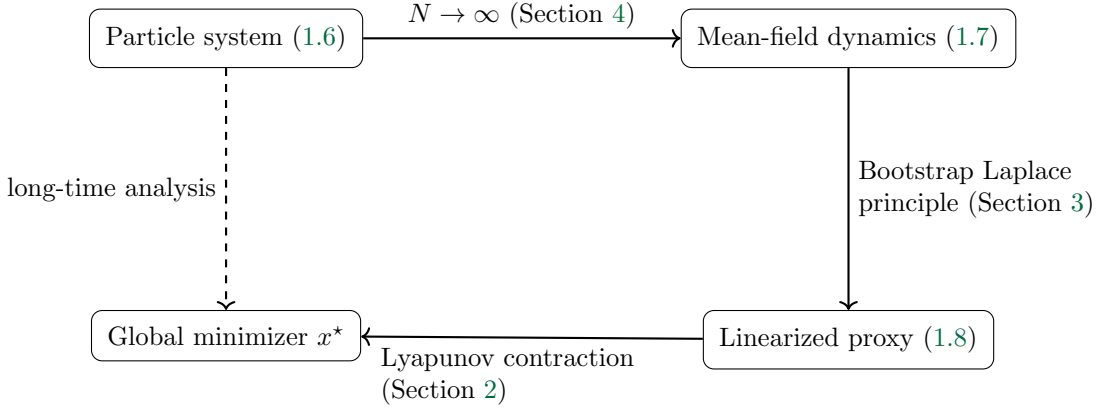

Altogether, the main results are as follows. First, \cref{t:minimizer} establishes the global convergence of the mean-field dynamics \eqref{eq:2nd-cbo-td-mf}. Specifically, it proves that for a sufficiently large $\alpha$ and a small ambient noise floor $\sigma_0$, the mean-field error functional \eqref{eq:err_functional} decays geometrically to any prescribed tolerance $\varepsilon$, satisfying:$$\mathrm{Err}(\overline{X}_n, \overline{V}_n) \le \max\left\{(1-\mu/2)^n\mathrm{Err}(\overline{X}_0, \overline{V}_0), \varepsilon\right\} \quad \text{for all } n \ge 0.$$
Next, \cref{thm: mfl-2nd-order-cbo} rigorously justifies the mean-field approximation via a synchronous coupling argument. It demonstrates that the finite-$N$ particle scheme tracks the mean-field trajectories over any finite horizon $n_T$ with a quantitative error rate of $\mathcal{O}(N^{-1/2})$, establishing that:$$\sup_{i=1,\dots,N} \E\left[\sup_{0 \le n \le n_T} \left(|X_n^i - \overline{X}_n^i| + |V_n^i - \overline{V}_n^i|\right)\right] \le C(n_T)N^{-1/2}.$$
Finally, \cref{thmmain} combines these estimates to provide an quantitative convergence guarantee for the finite-particle scheme. By combining the mean-field tracking error and the optimization tolerance at the iteration $n_T$, it bounds the expected squared distance between the particles' empirical mean and the true global minimizer $x^*$ as:$$\E\left[\left|\frac{1}{N}\sum_{i=1}^N X_{n_T}^i - x^*\right|^2\right] \le C_1\frac{1}{N} + C_2\varepsilon.$$ The analysis strategy and structure of the paper leading to these results are summarized in Figure \ref{fig:proof_strategy}.

\section{Contractivity of dynamics via discrete Lyapunov functional}\label{sec: contractivity}
The first step in our convergence analysis is to establish the mean-square stability of the time-discrete particle dynamics. For simple first-order gradient-like iterations, the squared Euclidean distance to the minimizer often provides a natural Lyapunov function under suitable assumptions and step-size conditions. In the second-order PSO dynamics considered here, however, the position error is coupled with the velocity, and stability of the position–velocity system does not generally imply a monotonic decrease of the Euclidean position error. Moreover, even when the linearized dynamics is asymptotically stable, its trajectories may exhibit transient growth or oscillatory behavior. 

To overcome this, we construct a quadratic Lyapunov functional adapted to the discrete map. We first linearize the mean-field dynamics around the global minimizer $x^\star$. In \cref{prop:h2_decay_dt1}, by using the discrete Lyapunov (Stein) theorem (\cref{thm:stein}), we construct a positive-definite matrix $P \succ 0$ associated with the system's transition matrix. Using the corresponding quadratic metric, we prove geometric decay for the nonlinear mean-field dynamics up to error terms arising from the consensus approximation and the noise.

\begin{lemma}[Stein's Theorem]\label{thm:stein}
    Let $A \in \mathbb{R}^{d \times d}$. The matrix $A$ is Schur stable (i.e., its spectral radius satisfies $\varrho(A) < 1$) if and only if for every symmetric positive definite matrix $Q \succ 0$, there exists a unique symmetric positive definite matrix $P \succ 0$ satisfying the discrete Lyapunov (Stein) equation:
    \begin{equation}
        A^\top P A - P = -Q.
    \end{equation}
\end{lemma}

\begin{proof}
    For the proof, we refer to \cite[Theorem 13.2.1]{lancaster1985theory} 
\end{proof}

\subsection{Energy decay}\label{subsec:h2-unit-dt} 

Let $(\overline X_n,\overline V_n)_{n\ge0}$ satisfy the mean-field dynamics \eqref{eq:2nd-cbo-td-mf}. The deterministic linear part of this system is governed, in the shifted variables, by the transition matrix
\[
    A := \begin{pmatrix}
    1-\frac{\lambda}{\gamma} & 1-\gamma\\
    -\frac{\lambda^2}{\gamma^2} & (1-\gamma)\left(1+\frac{\lambda}{\gamma}\right)
    \end{pmatrix}.
\]
We now identify the corresponding Schur-stable parameter regime. By the Jury criterion \cite{jury1964theory}, the roots of a real quadratic polynomial for $\tau, \delta \in \mathbb{R}\,,$
\[
    z^2-\tau z+\delta
\]
lie in the open unit disk if and only if
\[
    1-\delta>0,\qquad 
    1-\tau+\delta>0,\qquad 
    1+\tau+\delta>0.
\]
Applying this criterion to the characteristic polynomial
$
    p_A(z)=z^2-\operatorname{tr}(A)z+\det(A),
$
we compute
$    \operatorname{tr}(A)=2-\gamma-\lambda
$ and $
    \det(A)=1-\gamma.
$
Hence, substituting these expressions into the three Jury inequalities gives
\[
    \gamma>0,\qquad 
    \lambda>0,\qquad 
    4-2\gamma-\lambda>0\,.
\]
Equivalently, $\rho(A)<1$ if and only if
\begin{equation}
\label{eq: jury-condition-forA}
    0<\gamma<2,\qquad 0<\lambda<4-2\gamma.
\end{equation}
For any parameters $(\lambda,\gamma)$ in this Schur-stable regime, \cref{thm:stein} guarantees the existence of a unique symmetric positive definite matrix $\widetilde P\succ0$ solving $
    A^\top \widetilde P A-\widetilde P=-I_2.
$
Normalizing $\widetilde P$ by its first diagonal entry yields
\[
    P:=\frac{1}{\widetilde P_{11}}\widetilde P
    =:
    \begin{pmatrix}
    1 & \theta\\
    \theta & C_H
    \end{pmatrix}
    \succ0.
\]
In particular, that $P\in \R^{2\times 2}$ is positive semi-definite implies that $C_H>\theta^2$. By introducing the Kronecker product notation$$P \otimes I_d := \begin{pmatrix} I_d & \theta I_d \\ \theta I_d & C_H I_d \end{pmatrix} \in \mathbb{R}^{2d \times 2d},$$we can rewrite the error functional \eqref{eq:err_functional} as
\begin{equation}
 \err(X,V) = \expect\left[ Z^\top (P \otimes I_d) Z \right], \quad \text{with } Z := \binom{X-x^\star}{V-\frac{\lambda}{\gamma}(x^\star-X)} \in \mathbb{R}^{2d}.   
\end{equation}
With this adapted quadratic metric in hand, we now state the unit-step energy contraction estimate.

\begin{proposition}[Energy Functional decay] \label{prop:h2_decay_dt1}
    Assume the parameters $\lambda,\gamma$ satisfy \eqref{eq: jury-condition-forA}, and let $P$, $\theta$, $C_H$, and $\err(X,V)$ be defined as above. Let $q := 1/\widetilde{P}_{11} > 0$ and define the base deterministic contraction rate $\mu_0 := q / \lambda_{\max}(P) \in (0, 1)$. 
    Furthermore, for $b:=\binom{1}{1+\frac{\lambda}{\gamma}}$, assume the small-noise condition holds:
    \begin{equation}\label{eq:sigma_small-1}
        d\,\sigma^2\,\frac{b^\top P b}{\lambda_{\min}(P)}\le \frac{\mu_0}{12} \,.
    \end{equation}
    Then, there exist constants $\mu \in (0, 1)$ and $C_1 > 0$ such that for all $n\ge0$,
    \begin{equation}\label{eq:H2-recursion-1}
        \err(\overline X_{n+1},\overline V_{n+1})\ \le\ (1-\mu)\,\err(\overline X_n,\overline V_n)\ +\ C_1 \bigl(\sigma_0^2+|x^\alpha[\rho_n]-x^\star|^2\bigr).
    \end{equation}
\end{proposition}

\begin{proof}
    Introduce the error variables
    \[
    e_n:=\overline X_n-x^\star,\qquad
    u_n:=\overline V_n-\frac{\lambda}{\gamma}(x^\star-\overline X_n)
    =\overline V_n+ \frac{\lambda}{\gamma} e_n,\qquad
    Z_n:=\binom{e_n}{u_n}\in\R^{2d}.
    \]
    By construction, the functional evaluates to $\err(\overline X_n,\overline V_n)=\E\big[Z_n^\top(P\otimes I_d)Z_n\big]$. 

    Using $\overline V_n=u_n-\frac{\lambda}{\gamma} e_n$ and $x^\alpha[\rho_n]-\overline X_n=(x^\alpha[\rho_n]-x^\star)-e_n$, a direct algebraic computation from \eqref{eq:2nd-cbo-td-mf} yields the compact vector recursion:
    \begin{equation}\label{eq:Zn-recursion}
        Z_{n+1}=(A\otimes I_d)Z_n+\lambda\,(b\otimes I_d)(x^\alpha[\rho_n]-x^\star)
        +\sigma(\sigma_0+|x^\alpha[\rho_n]-\overline X_n|)\,(b\otimes I_d)\xi_n.
    \end{equation}

    Let $\mathcal F_n:=\sigma(\xi_0,\dots,\xi_{n-1})$. Since $\xi_n \in \R^d$ is independent of $\mathcal F_n$ with $\E[\xi_n|\mathcal F_n]=0$ and $\E[|\xi_n|^2|\mathcal F_n]=d$, the cross-terms involving the noise vanish. Expanding $Z_{n+1}^\top(P\otimes I_d)Z_{n+1}$ and taking the expectation yields the decomposition:
    \begin{align}
        \err(\overline X_{n+1},\overline V_{n+1})
        &=\E\big[Z_n^\top(A^\top P A\otimes I_d)Z_n\big]
        +\lambda^2(b^\top P b)\,|x^\alpha[\rho_n]-x^\star|^2
        \nonumber\\
        &\quad
        +2\lambda\,\E\Big[\big\langle (b^\top P A\otimes I_d)Z_n,\ x^\alpha[\rho_n]-x^\star\big\rangle\Big]
        +d\,\sigma^2(b^\top P b)\,\E\big[(\sigma_0+|x^\alpha[\rho_n]-\overline X_n|)^2\big].
    \label{eq:H2-split-1}
    \end{align}

    \smallskip
    \noindent\textbf{Step 1: Linear part via Stein's Theorem.}
    From the definition of our metric, $P$ satisfies $A^\top P A - P = -q I_2$. For any $z\in\R^{2d}$, this implies:
    \[
        z^\top(A^\top P A\otimes I_d)z = z^\top(P\otimes I_d)z - q|z|^2 \,.
    \]
    Using the Rayleigh quotient lower bound $|z|^2 \ge \frac{1}{\lambda_{\max}(P)} z^\top(P\otimes I_d)z$, we obtain:
    \[
        z^\top(A^\top P A\otimes I_d)z \le \Big(1 - \frac{q}{\lambda_{\max}(P)}\Big) z^\top(P\otimes I_d)z\,.
    \]
    We claim that $\mu_0:=\frac{q}{\lambda_{\max}(P)}\in(0,1)\,.$ The lower bound is immediate from $P\succ0$ and $q>0$. For the upper bound, note that $A^\top P A=P-qI_2\succeq0\,.$ Moreover, $A\neq0$, since $A_{21}=-\frac{\lambda^2}{\gamma^2}<0$. As $P\succ0$, this implies $A^\top P A\neq0$, and hence $
0 < \operatorname{tr}(A^\top P A) = \operatorname{tr}(P)-2q\,.$
Therefore, $q<\frac12\operatorname{tr}(P)
    \le \lambda_{\max}(P)\,,$
which proves that $\mu_0\in(0,1)$.
    Taking expectations yields a strict contraction for the linear part at the base rate $\mu_0$: 
    \begin{equation}\label{eq:lin_part_bound-1}
        \E\big[Z_n^\top(A^\top P A\otimes I_d)Z_n\big]\le (1-\mu_0)\,\err(\overline X_n,\overline V_n)\,.
    \end{equation}

    \smallskip
    \noindent\textbf{Step 2: Cross term.}
    Notice that the Euclidean norm of the cross-term multiplier decomposes as: $|(b^\top P A\otimes I_d)Z_n| = \|A^\top P b\|_2 \, |Z_n|$. We bound the cross term using Cauchy-Schwarz and the property $|Z_n|^2 \le \frac{1}{\lambda_{\min}(P)} Z_n^\top (P \otimes I_d) Z_n$. By Young's inequality, for any $\varepsilon>0$:
    \begin{align*}
        2\lambda\Big\langle (b^\top P A\otimes I_d)Z_n,\ x^\alpha[\rho_n]-x^\star\Big\rangle
        &\le 2\lambda \|A^\top P b\|_2 \, |Z_n| \, |x^\alpha[\rho_n]-x^\star| \\
        &\le \varepsilon \lambda_{\min}(P) |Z_n|^2
        +\frac{\lambda^2}{\varepsilon \lambda_{\min}(P)}\, \|A^\top P b\|_2^2\,|x^\alpha[\rho_n]-x^\star|^2.
    \end{align*}
    Taking expectations yields a bound relative to our error functional:
    \begin{equation}\label{eq:cross_bound-1}
        2\lambda\,\E\Big[\big\langle (b^\top P A\otimes I_d)Z_n,\ x^\alpha[\rho_n]-x^\star\big\rangle\Big]
        \le \varepsilon\,\err(\overline X_n,\overline V_n)+\frac{\lambda^2}{\varepsilon \lambda_{\min}(P)}\,\|A^\top P b\|_2^2\,|x^\alpha[\rho_n]-x^\star|^2.
    \end{equation}

    \smallskip
    \noindent\textbf{Step 3: Noise term.}
    Using the triangle inequality $|x^\alpha[\rho_n]-\overline X_n|=|(x^\alpha[\rho_n]-x^\star)-e_n|\le |x^\alpha[\rho_n]-x^\star|+|e_n|$ and $(x+y+z)^2\le 3(x^2+y^2+z^2)$, we have:
    \[
    (\sigma_0+|x^\alpha[\rho_n]-\overline X_n|)^2
    \le 3\big(\sigma_0^2+|x^\alpha[\rho_n]-x^\star|^2+|e_n|^2\big).
    \]
    Since the $P$-metric bounds the Euclidean norm, $\E|e_n|^2 \le \E|Z_n|^2 \le \frac{1}{\lambda_{\min}(P)}\err(\overline X_n,\overline V_n)$. Therefore:
    \begin{equation}\label{eq:noise_bound_2-1}
        d\,\sigma^2(b^\top P b)\,\E\big[(\sigma_0+|x^\alpha[\rho_n]-\overline X_n|)^2\big]
        \le \frac{3d\sigma^2(b^\top P b)}{\lambda_{\min}(P)}\err(\overline X_n,\overline V_n)+3d\sigma^2(b^\top P b)\big(\sigma_0^2+|x^\alpha[\rho_n]-x^\star|^2\big).
    \end{equation}

    \smallskip
    \noindent\textbf{Step 4: Final combination.}
    Inserting \eqref{eq:lin_part_bound-1}, \eqref{eq:cross_bound-1}, and \eqref{eq:noise_bound_2-1} into the expansion \eqref{eq:H2-split-1}:
    \begin{align*}
        \err(\overline X_{n+1},\overline V_{n+1})
        &\le \left(1-\mu_0+\varepsilon+\frac{3d\sigma^2(b^\top P b)}{\lambda_{\min}(P)}\right)\err(\overline X_n,\overline V_n)\\
        &\quad
        +\left(\lambda^2(b^\top P b)+\frac{\lambda^2}{\varepsilon \lambda_{\min}(P)}\| A^\top P b \|_2^2
        +3d\sigma^2(b^\top P b)\right)|x^\alpha[\rho_n]-x^\star|^2
        +3d\sigma^2(b^\top P b)\sigma_0^2.
    \end{align*}
    We choose the Young's inequality parameter $\varepsilon := \mu_0/4$. Under the small-noise condition \eqref{eq:sigma_small-1}, the noise contribution to the Lyapunov multiplier, $\frac{3d\sigma^2 (b^T P b)}{\lambda_{\min} (P)}$ is bounded by $\mu_0/4$. Thus, the total multiplier is bounded by $1 - \mu_0 + \mu_0/4 + \mu_0/4 = 1 - \mu_0/2$. 
    Setting the decay rate $\mu := \mu_0/2$ and the constant 
    \[
        C_1 := \lambda^2(b^\top P b)+\frac{4\lambda^2}{\mu_0 \lambda_{\min}(P)}\| A^\top P b \|_2^2+3d\sigma^2(b^\top P b)\,,
    \]
    we obtain \eqref{eq:H2-recursion-1} as claimed.
\end{proof}

\begin{remark}[The choice of metric $P$ and the Schur-stable regime]
\label{rmk:metric-choice}
    In defining our functional $\err(X,V)$, we use a positive definite matrix
    $P \succ 0$ constructed via Stein's theorem. A natural simplification would be
    to work directly with the standard Euclidean distance by setting $P=I_2$
    (i.e., $C_H=1$ and $\theta=0$). However, imposing $P=I_2$ restricts the admissible parameter regime. In Step 1
    of the proof, contraction of the deterministic linear part requires the Lyapunov
    dissipation matrix $A^\top P A-P$
    to be strictly negative definite. If one imposes $P=I_2$, this becomes $A^\top A-I_2\prec0\,,$
    which is equivalent to the one-step Euclidean contraction condition $\|A\|_2<1\,.$
    For second-order swarm dynamics, this condition is stronger than Schur stability.
    Indeed, the Schur-stable regime $\varrho(A)<1$ contains parameter values for
    which the transition matrix $A$ is non-normal and satisfies $\|A\|_2>1$, even
    though all eigenvalues of $A$ lie strictly inside the unit disk. To see where
    oscillations enter, recall that the characteristic polynomial of the unit-step
    transition matrix is
    \[
        p_A(z)
        =
        z^2-\operatorname{tr}(A)z+\det(A),
        \qquad
        \operatorname{tr}(A)=2-\gamma-\lambda,
        \quad
        \det(A)=1-\gamma .
    \]
    Hence its discriminant is
    \[
        \operatorname{tr}(A)^2-4\det(A)
        =
        (2-\gamma-\lambda)^2-4(1-\gamma)
        =
        (\gamma+\lambda)^2-4\lambda .
    \]
    Whenever this quantity is negative, the eigenvalues of $A$ form a complex
    conjugate pair, and the dynamics converge in an oscillatory way.
    This is the discrete-time counterpart of the underdamped regime for the
    continuous-time damped oscillator
    \[
        \ddot X(t)+\gamma\dot X(t)+\lambda(X(t)-x^\star)=0,
    \]
    whose characteristic polynomial $s^2+\gamma s+\lambda$ has complex conjugate
    roots when $\gamma^2<4\lambda$. In this oscillatory regime, the Euclidean distance to the minimizer may increase
    transiently as particles spiral and exchange kinetic and potential energy.
    Therefore, an iterative contraction argument based on $P=I_2$ breaks down.
    By contrast, Stein's theorem provides a quadratic form adapted to the discrete
    map, allowing us to capture the Schur-stable regime. 
\end{remark}

\begin{remark}[Explicit formula for $q$]
Solving the Stein equation $A^\top \widetilde P A - \widetilde P = -I_2$ yields
\[
    q = \frac{1}{\widetilde P_{11}}
    = \frac{\gamma^5 \lambda\,(2\gamma + \lambda - 4)}{D(\gamma, \lambda)},
\]
where $D(\gamma, \lambda)$ is a polynomial of degree $7$ in $(\gamma, \lambda)$. The numerator contains the factor $2\gamma + \lambda - 4$, so $q \to 0$ as $(\gamma, \lambda)$ approaches the Schur boundary $\lambda = 4 - 2\gamma$. This is consistent with \cref{fig:decay-rate-heatmap}, where the boundary that plots the contraction rate $\mu_0$ degenerates.
\end{remark}

\subsection{Energy decay in the small-step regime}
\label{subsec:h2-small-dt}
While \cref{prop:h2_decay_dt1} treats the algorithmically relevant unit-step regime, it is also useful to keep the time-step parameter visible in order to compare the discrete dynamics with the continuous-time second-order PSO/CBO model that motivates the construction. The deterministic part of the $\Delta t$-dependent update is a first-order discretization of the linearized second-order ODE
\begin{equation}
    \ddot{X}(t) + \gamma \dot{X}(t) + \lambda\bigl(X(t)-x^\star\bigr)=0.
\end{equation}
At the continuous-time level, the linearized flow is stable for every $\gamma>0$ and $\lambda>0$; the condition $\gamma^2<4\lambda$ corresponds to the underdamped regime in which convergence is oscillatory. By contrast, at the unit-step scale the relevant object is the discrete map, whose Schur stability is equivalent to the stricter condition \eqref{eq: jury-condition-forA} used in \cref{prop:h2_decay_dt1}. These additional restrictions are therefore from the discretiazation.

The purpose of this subsection is to show that the tilted-metric construction is consistent with the continuous-time stability mechanism in the vanishing step-size limit. To make this precise, the deterministic linear transition matrix at step size $\Delta t > 0$ takes the form
\begin{align}
\label{eq:matrix-A-small}
    A(\Delta t) = \begin{pmatrix} 
    1 - \frac{\lambda}{\gamma}\Delta t & \Delta t(1-\gamma\Delta t) \\ 
    -\left(\frac{\lambda}{\gamma}\right)^2 \Delta t & (1-\gamma\Delta t)\left(1+\frac{\lambda}{\gamma}\Delta t\right) 
    \end{pmatrix},
\end{align}
which reduces to the unit-step matrix $A$ of \cref{prop:h2_decay_dt1} when $\Delta t = 1$. For a generic positive definite metric
$
    P = \begin{pmatrix} 1 & \theta \\ \theta & C_H \end{pmatrix} \succ 0,
$
\cref{lem: dissipation-expansion} gives the Taylor expansion
\begin{equation}
\label{eq:A-Taylor}
    A(\Delta t)^\top P A(\Delta t) = P - \Delta t\, Q_0 + O(\Delta t^2),
\end{equation}
where the infinitesimal dissipation matrix $Q_0$ is independent of $\Delta t$. As derived in \cref{lem: dissipation-expansion}, $Q_0$ is symmetric positive definite, $Q_0 \succ 0$, for every $\lambda,\gamma>0$ provided that $C_H$ and $\theta$ satisfy
\begin{equation}
\label{eq:continuous-conditions}
    \theta = \frac{1}{\gamma} - \frac{C_H\lambda^2}{\gamma^3} 
    \quad \text{and} \quad 
    \frac{1}{\gamma^2 - \lambda + \lambda^2/\gamma^2} < C_H < \frac{\gamma^4}{\lambda^3} + \frac{\gamma^2}{\lambda^2}.
\end{equation}
Note that by \cref{lem: dissipation-expansion}, the interval in \eqref{eq:continuous-conditions} is nonempty for every $\lambda,\gamma>0$. Once $Q_0 \succ 0$, the remainder in \eqref{eq:A-Taylor} is dominated for all sufficiently small $\Delta t$, yielding discrete-time dissipation at rate $O(\Delta t)$ for the linearized part. After the same small-noise absorption used in the unit-step argument, this gives the energy recursion stated in \cref{prop: h2-decay-small-dt} below. In this sense, the small-step result in \cref{prop: h2-decay-small-dt} complements the unit-step Schur analysis by showing that our discrete Lyapunov construction recovers this continuous-time dissipation mechanism as $\Delta t \to 0$.

\begin{remark}[Distinction from the unit-step Stein matrix]
\label{rmk:small-step-metric-choice}
    The matrix used in this subsection and the one in \cref{prop:h2_decay_dt1} for the unit-step map are not identical. here, we use an abuse of notation for notational consistency. To be more specific, in the unit-step argument, $\widetilde P$ is uniquely determined as the solution of the Stein equation
    \[
        A^\top \widetilde P A - \widetilde P = -I_2
    \]
    with $A = A(1)$, and the entries $C_H$ and $\theta$ in the normalized matrix $P = \widetilde P / \widetilde P_{11}$ are then computed quantities. Here, by contrast, the same parametrization
    $
        P = \begin{pmatrix} 1 & \theta \\ \theta & C_H \end{pmatrix}
    $
    is used in \cref{prop: h2-decay-small-dt} fir snakk $\Delta t$, and there, $C_H, \theta$ are \emph{free parameters} chosen so that the leading-order dissipation matrix $Q_0$ in the expansion
    \[
        A(\Delta t)^\top P A(\Delta t) = P - \Delta t\, Q_0 + O(\Delta t^2)
    \]
    is positive definite. The admissible region \eqref{eq:continuous-conditions} typically defines a one-parameter family of such metrics rather than a unique solution. Consequently, although the same notation $C_H, \theta$ is used in both subsections for convenience of notation, the values of these parameters in the small-step regime need not coincide with those obtained from \cref{prop:h2_decay_dt1}.
\end{remark}

\begin{proposition}[Energy Functional decay with small $\Delta t$]\label{prop: h2-decay-small-dt}
    Let $\lambda > 0$ and $\gamma > 0$. Choose $C_H$ and $\theta$ satisfying the continuous-time stability conditions \eqref{eq:continuous-conditions}, such that $P = \begin{pmatrix} 1 & \theta \\ \theta & C_H \end{pmatrix} \succ 0$ and $Q_0 \succ 0$.
    
    Then, for the error functional $\err(\overline X_n, \overline V_n)$ defined via $P$, there exist constants $\mu_0 > 0, \Delta t_0>0$ and $K < \infty$ such that for sufficiently small noise $\sigma$ and every \(0<\Delta t\leq \Delta t_0\),
    \begin{equation}\label{eq:H2-lemma-small-dt-form}
        \err(\overline X_{n+1},\overline V_{n+1})\;\le\;(1-\mu_0\Delta t)\,\err(\overline X_n, \overline V_n)
        \;+\;K\Delta t\Big(\sigma_0^2 + \E\big|x^\alpha[\rho_n]-x^\star\big|^2\Big).
    \end{equation}
\end{proposition}

\begin{proof}
    Proof can be found in \cref{pf: h2-small-dt}.
\end{proof}



\subsection{Numerical validation of contractivity and stability}\label{sec: numerics}

\begin{figure}[t]
\centering

\newcommand{\panelwidth}{6cm}
\newcommand{\panelheight}{5.3cm}

\begin{subfigure}[t]{0.48\textwidth}
\centering
\begin{tikzpicture}
\begin{axis}[
    width=\panelwidth,
    height=\panelheight,
    scale only axis,
    xmin=0, xmax=2.5,
    ymin=0, ymax=5,
    axis lines=box,
    xlabel={$\gamma$ (friction)},
    ylabel={$\lambda$ (attraction)},
    xtick={0,1,2},
    ytick={0,1,2,3,4,5},
    clip=true
]

\node at (axis cs:0.35,0.5) {STABLE};
\node at (axis cs:2,4.5) {UNSTABLE};


\addplot[very thick, solid, domain=0:2, samples=2]
    {(4 - 2*x*1.0)/(1.0^2)}
    node[pos=0.55, sloped, below] {$\Delta t=1$};

\pgfmathsetmacro{\dt}{0.8}
\addplot[very thick, dashed, domain=0:2.5, samples=2]
    {(4 - 2*x*\dt)/(\dt^2)}
    node[pos=0.55, sloped, below] {$\Delta t=\pgfmathprintnumber{\dt}$};

\pgfmathsetmacro{\ddt}{0.7}
\addplot[very thick, dotted, domain=0:2.5, samples=2]
    {(4 - 2*x*\ddt)/(\ddt^2)}
    node[pos=0.65, sloped, below] {$\Delta t=\pgfmathprintnumber{\ddt}$};

\addplot[very thick, dotted, domain=0:2.5, samples=2]
    {(4 - 2*x*0.1)/(0.1^2)}
    node[pos=0.5, sloped, above] {$\Delta t=0.1$};

\end{axis}
\end{tikzpicture}
\caption{Stability boundaries.}
\label{fig:stability-dt-single}
\end{subfigure}
\hfill
\begin{subfigure}[t]{0.48\textwidth}
\centering
\begin{tikzpicture}
\begin{axis}[
    width=\panelwidth,
    height=\panelheight,
    scale only axis,
    xmin=0, xmax=2.1,
    ymin=0, ymax=4.1,
    axis lines=box,
    xlabel={$\gamma$ (friction)},
    ylabel={$\lambda$ (attraction)},
    xtick={0,0.5,1,1.5,2},
    ytick={0,1,2,3,4},
    clip=true,
    colormap/viridis,
    point meta min=0,
    point meta max=0.75,
    colorbar,
    colorbar style={
        ylabel={$\mu_0 = q/ \lambda_{\max}(P)$}
    }
]

\addplot[
    matrix plot*,
    mesh/cols=61,
    point meta=explicit,
    draw=none,
    unbounded coords=discard
] table[
    x=gamma,
    y=lambda,
    meta=mu
] {Numerics/decay_rate_table.dat};


\end{axis}
\end{tikzpicture} 
\caption{Decay rate ($\Delta t = 1$).}
\label{fig:decay-rate-heatmap}
\end{subfigure} \hfill

\caption{
Stability and deterministic contraction properties in the $(\gamma,\lambda)$-plane.
Plot \subref{fig:stability-dt-single} shows the analytic stability boundaries
$    \lambda=(4-2\gamma\Delta t)/\Delta t^2$
for selected values of $\Delta t$. The stable region lies below the corresponding boundary, while the unstable region lies above it. 
Plot \subref{fig:decay-rate-heatmap} shows the Stein deterministic contraction rate
$\mu_0=q/\lambda_{\max}(P)$ over the Schur-stable triangle for $\Delta t=1$.
Here $P=\widetilde P/\widetilde P_{11}$, where
$\widetilde P$ solves
$A^\top \widetilde P A-\widetilde P=-I_2$, and
$q=1/\widetilde P_{11}$.
}
\label{fig:stability-and-decay}
\end{figure}
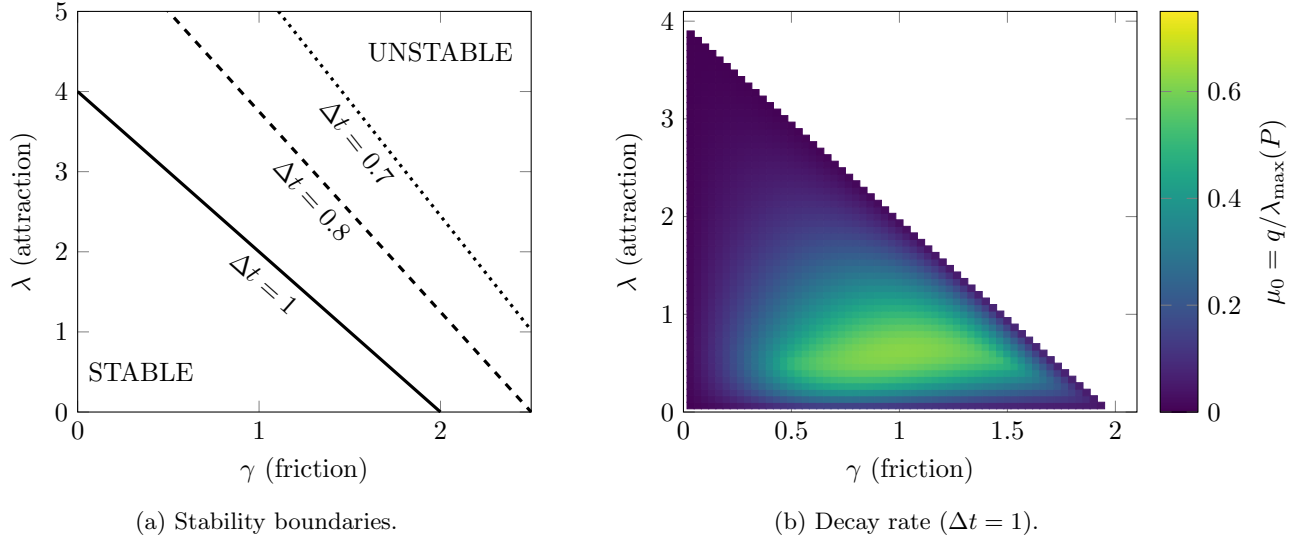
In this section, we illustrate the stability analysis from \cref{subsec:h2-unit-dt} and \cref{subsec:h2-small-dt} with numerics.  

\paragraph{Discrete versus continuous stability ranges.}
\Cref{fig:stability-dt-single} shows the Schur-stable region of the linearized transition matrix $A(\Delta t)$ defined in \eqref{eq:matrix-A-small} for
\[
    \Delta t \in \{0.7,\,0.8,\,1\}.
\]
For smaller step sizes, the Schur-stable region expands substantially in the original $(\gamma,\lambda)$-coordinates illustrating that it approahces the continuous time stability regime $\gamma > 0$, $\lambda > 0$ as $\Delta t \to 0$. For the algorithmically relevant unit step $\Delta t = 1$, the discrete map imposes the bounded Schur triangle
$ 0 < \gamma < 2,$ $0 < \lambda < 4 - 2\gamma$, as derived in \cref{prop:h2_decay_dt1}.

\begin{figure}[t]
    \centering
    \includegraphics[width=\textwidth]{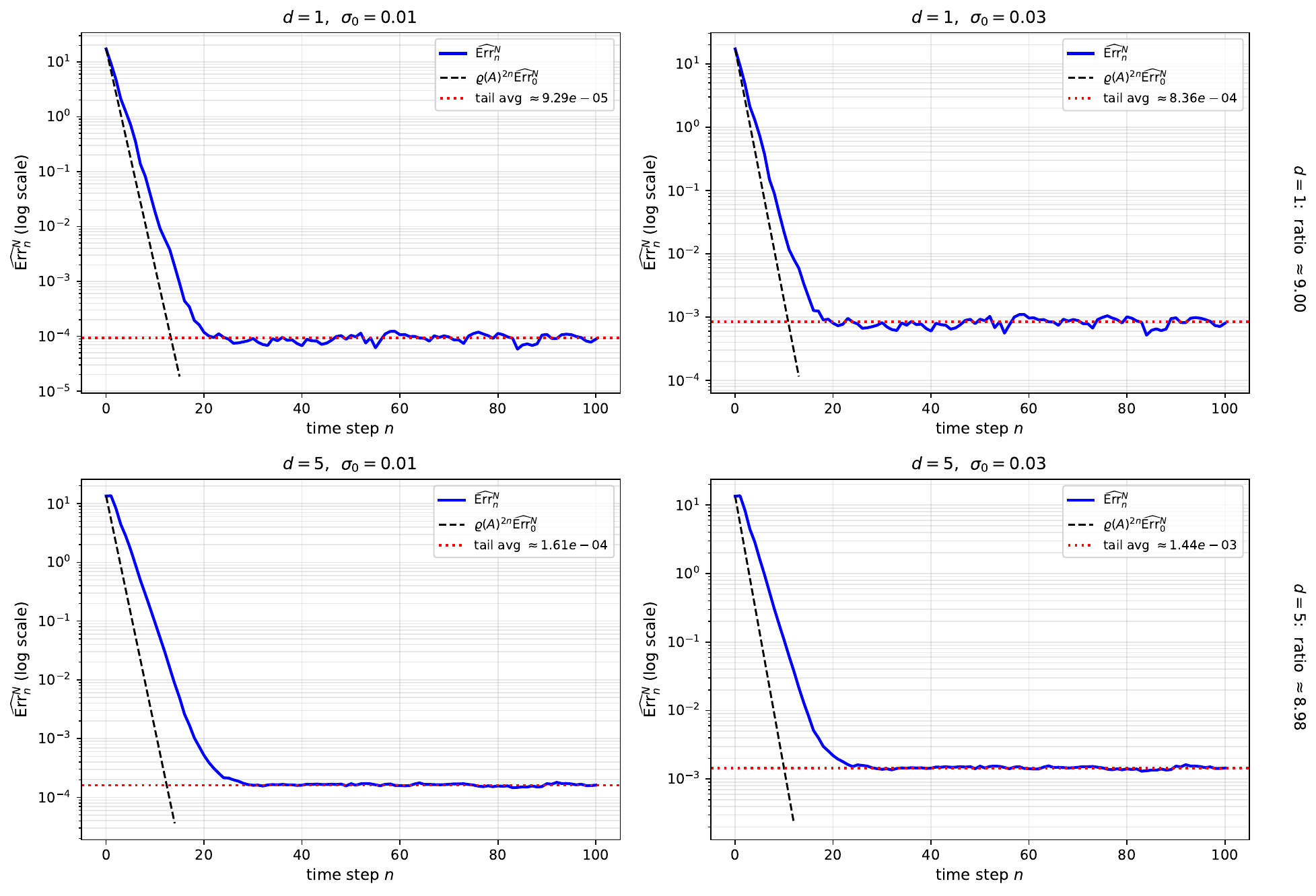}
    \caption{Empirical tilted-metric energy $\widehat{\mathrm{Err}}_n^N$
    along finite-particle simulations of \eqref{eq:2nd-cbo-td-particle} on
    the Ackley benchmark. \textbf{Top row:} dimension $d = 1$. \textbf{Bottom row:}
    dimension $d = 5$. \textbf{Left column:} noise floor $\sigma_0 = 0.01$.
    \textbf{Right column:} noise floor $\sigma_0 = 0.03$. In each panel the
    dashed black line is the deterministic reference slope
    $\varrho(A)^{2n}\widehat{\mathrm{Err}}_0^N$ and the red dotted line is
    the empirical tail average over the last $30$ time steps. The ratio of
    the two tail averages within each row is $\approx 9.00$ ($d = 1$) and
    $\approx 8.98$ ($d = 5$), in close agreement with the $\mathcal{O}(\sigma_0^2)$
    remainter term predicted by \cref{prop:h2_decay_dt1}. Parameter values
    are listed in \cref{app:parameters}.}
    \label{fig:log-energy-decay}
\end{figure}

\paragraph{Energy decay in the tilted $P$-metric.}
We now illustrate the geometric decay of the empirical tilted-metric energy
along trajectories of the finite-particle stochastic dynamics
\eqref{eq:2nd-cbo-td-particle}. For each particle $i = 1,\dots,N$ we work
in the shifted coordinates
\[
    Z_n^i := \begin{pmatrix} e_n^i \\ u_n^i \end{pmatrix} \in \R^{2d},
    \qquad
    e_n^i := X_n^i - x^\star,
    \qquad
    u_n^i := V_n^i + \frac{\lambda}{\gamma} e_n^i,
\]
introduced in the proof of \cref{prop:h2_decay_dt1}, and consider the empirical analogue of the energy functional \eqref{eq:err_functional},
\[
    \widehat{\mathrm{Err}}_n^N
    := \frac{1}{N} \sum_{i=1}^N (Z_n^i)^\top (P \otimes I_d)\, Z_n^i.
\]
 
\Cref{fig:log-energy-decay} reports $\widehat{\mathrm{Err}}_n^N$ for two
dimensions $d \in \{1, 5\}$ and two values of the noise floor
$\sigma_0 \in \{0.01,\,0.03\}$, with all other parameters fixed within each
dimension. In all four panels the empirical energy decays geometrically
along the deterministic reference slope
$\varrho(A)^{2n}\widehat{\mathrm{Err}}_0^N$ until it saturates at a level
determined by the non-degenerate noise. The figure isolates the role of the noise floor across two dimensions: the
deterministic decay rate which depends only on Stein metric $P$, spectral radius $\varrho(A)$,
initial state distribution is fixed within each row, so the only effect
of varying $\sigma_0$ is to change the level at which the geometric decay
is affected by the noise.

\section{Global convergence via the Laplace principle}\label{sec: convergence}
The goal of this section is to close the analytical loop using a bootstrap induction argument. The challenge lies in the fact that the contractivity of the swarm depends on the accuracy of the consensus point, but the accuracy of the consensus point depends on the concentration of the swarm. We break this circular dependence in two steps. First, we prove a mass concentration lower bound, demonstrating that the non-degenerate multiplicative noise guarantees a strictly positive fraction of particles remains near $x^\star$ at all times. Second, after bounding the mass, we use the quantitative Laplace principle to prove that the consensus point $x^\alpha[\rho_n]$ becomes exponentially close to $x^\star$. By recursively applying these bounds, we show that the the residual term remains controlled, allowing the geometric contraction from \cref{sec: contractivity} to drive the swarm to the global minimizer.

We study the long-time behavior of the mean-field system \eqref{eq:2nd-cbo-td-mf} to understand under which assumptions it converges to a global solution to \eqref{eq:problem}. The analysis is based on the application of the quantitative Laplace principle \cite[Proposition 4.5]{fornasier2024convergence} and on showing that the dynamics is contractive. First, we state our assumptions on the objective function.
\begin{assumption} \label{asm:inverse}
The objective function $\calE \in C(\R^d)$ satisfies:
\begin{enumerate}
\item there exists  a unique $x^\star$ such that $\calE(x^\star) = \inf_{x \in \R^d} \calE(x)$;
\item there exists $\calE_{\infty}, R_0, \eta,\nu>0$ such that 
\begin{align}
    |x - x^\star| &\leq \frac1\eta \left(\calE(x) - \inf \calE \right)^{\nu} \qquad \textup{for all}\;\; x\in B(x^\star, R_0) \\
    \calE(x) - \inf \calE &> \calE_{\infty} \qquad \textup{for all}\;\; x\in \R^d \setminus B(x^\star, R_0)\,.
\end{align}
\end{enumerate}
\end{assumption}

\subsection{Application of quantitative Laplace principle}

As the parameter $\alpha$ used to compute the consensus point $x^\alpha[\rho]$ \eqref{eq:consensuspoint} increases, we can expect it to convergence towards the global minimum $x^\star$, provided it belongs to the support of $\rho$. Under \cref{asm:inverse}, it was derived in \cite{fornasier2024convergence} a quantitative convergence rate.

\begin{proposition}[{\cite[Proposition 4.5]{fornasier2024convergence}}] \label{prop:laplace}
    Let $\rho \in \mathcal{P}(\R^d)$ and fix $\alpha>0$. For any $e>0$, define $\calE_e := \sup_{x\in B(x^\star,e)} \calE(x)$, then, under the inverse continuity property \cref{asm:inverse} and assuming w.l.o.g. $\inf \calE = 0$, for any $r\in (0, R_0]$ and $q>0$ such that $q + \calE_r < \calE_{\infty}$, we have
    \begin{equation}
        |x^\alpha[\rho] - x^\star | \leq \frac{(q + \calE_r)^\nu}{\eta} + \frac{\exp(-\alpha q)}{\rho(B(x^\star,r))}\int |x - x^\star|\rho(d x)\,.
    \end{equation}
\end{proposition}

Note that the bound is finite provided $\rho(B(x^\star, r))>0$, that is, $x^\star$ belongs to the support of $\rho$. Moreover, to iteratively apply this bound for $\rho_n$ we need to provide an estimate on the mass around the solution for all iterations $n\geq 1$ until convergence. Using a similar strategy to the one derived in \cite{borghi2025bgk}, we achieve it by leveraging the fact that the noise is non-degenerate in \eqref{eq:2nd-cbo-td-mf}.

\begin{lemma} \label{lem:mass_laplace}
Let $\mu, C_1$ be the constants from \cref{prop:h2_decay_dt1} and \cref{asm:inverse} hold. Assume that the parameters $\gamma, \lambda, \sigma$ satisfy the conditions in \cref{prop:h2_decay_dt1}.  Additionally assume for a given $n\in \mathbb{N}$, $E_0>0$, $\veps\in (0, E_0]$, and $\sigma_0^2 \leq \veps \mu/4 C_1$, it holds 
\begin{equation*}
\begin{dcases}
\err(\OX_n, \OV_n) &\leq E_0 \\
|x^\alpha[\rho_n] - x^\star|^2 &\leq \veps \frac{\mu}{4C_1}\,.
\end{dcases}
\end{equation*}    
Then, there exists $\alpha = \alpha(E_0,\veps)> 0$ sufficiently large such that 
\[|x^\alpha[\rho_{n+1}] - x^\star|^2 \leq \veps \frac{\mu}{4C_1}\,.\]
\end{lemma}
\begin{proof} 
\textbf{Mass around large ball $B(0,R)$.}
We start by estimating the mass at $B(0,R)$ for some large radius $R, R>|x^\star|, R>\sqrt{\veps \mu/(2C_1)}$. Since by Markov's inequality applied to the squared norm we have 
\[\proba\left(|\OX_n| + |\OV_n| > R \right) \leq \proba\left((|\OX_n| + |\OV_n|)^2 > R^2 \right)
\leq \frac1{R^2}\expect\left[(|\OX_n| + |\OV_n|)^2 \right] \leq \frac1{R^2} C(\lambda,\gamma, E_0, |x^\star|)
\]
therefore, for any given $\delta>0$ there exist $R$ sufficiently large such that for $\mathcal{F}_n^R := \{|\OX_n|, |\OV_n|\leq R\}$
\[
\proba(\mathcal{F}_n^R) \geq 
\proba \left(|\OX_n|+ |\OV_n|\leq R \right) > 1 - \delta\,.
\]

\noindent
\textbf{Mass around small ball $B(x^\star, r)$.}
We aim to prove
$
\proba\left(\OX_{n+1}\in B(x^\star,r)\right)\ \ge\ \delta_r> 0\,.
$
From the update rule \eqref{eq:2nd-cbo-td-mf}, we have
\begin{align*}
\OX_{n+1}-x^\star
&= 
\big(\OX_n - x^\star\big) + (1-\gamma)\OV_n
+ \lambda  \big(x^\alpha[\rho_n]-\OX_n\big)
+\sigma \big(\sigma_0 + |x^\alpha[\rho_n]-\OX_n|\big)\,\xi_n \\
&  =: Z_1 + Z_2\xi_n\,,
\end{align*}
so that
$
\{\OX_{n+1}\in B(x^\star,r)\}=\{\,Z_1 + Z_2\xi_n \in B(0,r)\,\}.
$
For any $\omega \in \mathcal F_n^R$, due to the assumption on $x^\alpha[\rho_n]$ and the choice $R$ we have 
\[
|x^\alpha[\rho_n] - \OX_n(\omega)|
 \leq  |x^\alpha[\rho_n] - x^\star| + |x^\star| + | \OX_n(\omega)| \leq 3R\,.
\]
This leads to an upper and lower bound on $Z_2 = \sigma(\sigma_0 + |x^\alpha[\rho_n] - \OX_n(\omega)|)$ for some constants $c_{2,\min}, c_{2,\max}>0$
\begin{equation}\label{eq:c2-bds}
c_{2,\min}:=\sigma\,\sigma_0 \le\ Z_2(\omega)\ \le\ \sigma (\sigma_0+3R)=:c_{2,\max}\qquad \textup{for all}\quad \omega \in \mathcal{F}_n^R\,.
\end{equation}
Moreover, using similar estimates we get for all $\omega \in \mathcal{F}_n^R$
\begin{equation}\label{eq:c1-bd}
\begin{split}
|Z_1(\omega)| &
\le |\OX_n(\omega)-x^\star| + (1-\gamma)|\OV_n(\omega)|
    + \lambda  |x^\alpha[\rho_n]-\OX_n(\omega)| \\
& \le |x^\star| +  \bigl(1+|1-\gamma|+3\lambda\bigr)R 
=: C_R.
\end{split}
\end{equation}

First, we note that 
\[
\proba\left(\OX_{n+1} \in B(x^\star, r)\right) = \proba\left( |\OX_{n+1} - x^\star| \leq r \right) = 
\proba\left( |Z_1 + Z_2 \xi_n| \leq r \right) = \proba\left( Z_1 + Z_2\xi_n \in B(0,r)\right)\,.
\]
Let $\textbf{1}[\cdot]$ be the indicator function. 
We notice that
\begin{align}\label{2.9}
    &\proba\left( Z_1 + Z_2\xi_n \in B(0,r)\right)\geq \proba\left(\mathcal{F}_n^R\cap (Z_1 + Z_2\xi_n \in B(0,r))\right)\notag\\
    &=\expect\big[ 
\textbf{1}[\mathcal{F}_n^R]\textbf{1}[Z_1 + Z_2 \xi_n \in B(0,r)] \big]\notag\\
&=\expect\Big[\expect\big[
\textbf{1}[\mathcal{F}_n^R]\textbf{1}[Z_1 + Z_2 \xi_n \in B(0,r)] |\OX_n,\OV_n\big]\Big] \qquad \mbox{(Law of total expectation)}\notag\\
&=\expect\Big[\textbf{1}[\mathcal{F}_n^R]\expect\left[
\textbf{1}[Z_1 + Z_2 \xi_n \in B(0,r)] |\OX_n,\OV_n\right]\Big]\qquad  \mbox{(}\textbf{1}[\mathcal{F}_n^R] \mbox{ is } \sigma(\OX_n,\OV_n)-\mbox{measurable)}\notag\\
   & =\expect\left[ \textbf{1}[\mathcal{F}_n^R] \proba\left( Z_1(\OX_n,\OV_n) + Z_2(\OX_n,\OV_n)\xi_n \in B(0,r)|\OX_n,\OV_n\right) \right]
\end{align}
and
\begin{equation}\label{2.10}
\proba\left(Z_1 + Z_2\xi_n \in B(0,r)|\OX_n=x,\OV_n=v\right)=\int_{B(0,r)} \frac1{(2\pi)^{d/2} z_2(x,v)^d} \exp\left( - \frac{|z_1(x,v) - z|^2}{2z_2(x,v)^2}
\right) dz\,.
\end{equation}
Then using \eqref{2.9} and \eqref{2.10} we have
\begin{align}
    \proba\Big( &Z_1 + Z_2\xi_n \in B(0,r)\Big)
    \geq \expect\left[ \textbf{1}[\mathcal{F}_n^R] \proba\left( Z_1(\OX_n,\OV_n) + Z_2(\OX_n,\OV_n)\xi_n \in B(0,r)|\OX_n,\OV_n\right) \right]\notag\\
&\geq \iint_{B_R\times B_R}\textbf{1}[\mathcal{F}_n^R](x,v)\int_{B(0,r)} \frac1{(2\pi)^{d/2} z_2(x,v)^d} \exp\left( - \frac{|z_1(x,v) - z|^2}{2z_2(x,v)^2}
\right) dz f^n(dx,dv) \,,
\end{align}
where $f^n$ is the joint distribution of $(\OX_n,\OV_n)$. Recall that when $(x,v)\in B_R\times B_R$, it holds that $|z_1|\leq C_R$ and $c_{2,\min}\leq z_2\leq c_{2,\max}$. This further implies that
\begin{align}
    \proba\left( Z_1 + Z_2\xi_n \in B(0,r)\right)&\geq \frac{|B(0,r)|}{(2\pi)^{d/2} c_{2,\max}^d} \exp\left( - (r+C_R)^2/2c_{2,\min}^2\right)
\iint_{B_R\times B_R}\textbf{1}[\mathcal{F}_n^R](x,v) f^n(dx,dv)\notag\\
&=\proba(\mathcal{F}_n^R) \frac{|B(0,r)|}{(2\pi)^{d/2} c_{2,\max}^d} \exp\left( - (r+C_R)^2/2c_{2,\min}^2
\right)\,.
\end{align}
By using the fact that $\proba(\mathcal{F}_n^R)\geq 1 - \delta$ we finally obtain that $\proba \left(\OX_{n+1}\in B(x^\star,r)\right)\ \ge\ \delta_r> 0$ for some $\delta_r>0$.

\noindent
\textbf{Estimate on $\int |x - x^\star|\rho_{n+1}(dx)$.} By definition of the error functional \eqref{eq:err_functional}, we have $\expect|\OX_{n+1} - x^\star|^2 \leq  \err(\OX_{n+1}, \OV_{n+1})/\lambda_{\min}(P)$. Therefore, by using Jensen's inequality, the contractivity of the dynamics from \cref{prop:h2_decay_dt1}, and the assumption on $|x^\alpha[\rho_n] - x^\star|^2$, $\err(\OX_n, \OV_n)$ and $\veps$, we have
\begin{align*}
\err(\OX_{n+1}, \OV_{n+1}) &\leq (1 - \mu) \err(\OX_{n}, \OV_{n}) + C_1\left( \sigma_0^2 + |x^\alpha[\rho_n]-x^\star|^2 \right) \\
& \leq (1 - \mu) E_0 + C_1\left( \veps \frac{\mu}{4C_1} + \veps \frac{\mu}{4C_1} \right) \\
& \leq (1 - \mu/2) E_0 \leq E_0
\end{align*}
leading to $\int |x - x^\star|\rho_{n+1}(dx) \leq \sqrt{\expect|\OX_{n+1} - x^\star|^2} \leq \sqrt{E_0/\lambda_{\min}(P)}$.

\noindent
\textbf{Application quantitative Laplace principle.}
Next, we take a radius $r\in(0,1]$ and $q>0$ sufficiently small such that (in the notation of \cref{prop:laplace}) it holds
\[
\frac{(q + \calE_r)^{\nu}}\eta < \frac{1}{2}\sqrt{\veps \frac{\mu}{4 C_1}}
\]
By applying \cref{prop:laplace} with such $r$, 
we then have 
\begin{align*}
| x^\alpha[\rho_{n+1}] - x^\star| & \leq  \frac{1}{2}\sqrt{\veps \frac{\mu}{4C_1}} + \frac{\exp(-\alpha q)}{\delta_{r}}\int |x - x^\star|\rho_{n+1}(dx) \\
& \leq  \frac{1}{2}\sqrt{\veps \frac{\mu}{4C_1}} + \frac{\exp(-\alpha q)}{\delta_{r}}\sqrt{\frac{E_0}{\lambda_{\min}(P)}}  \leq \sqrt{\veps \frac{\mu}{4C_1}}
\end{align*}
for $\alpha = \alpha(E_0,\veps)>0$ sufficiently large such that the second term is also bounded by $\frac{1}{2}\sqrt{\veps \mu/(4C_1)}$. Squaring this result yields $|x^\alpha[\rho_{n+1}] - x^\star|^2 \leq \veps \mu/(4C_1)$. 
\end{proof}

\begin{remark} \label{rmk:constants1}
The smallness condition on $\sigma_0$ and the largeness condition on $\alpha$ are linked through the lower bound on the mass near $x^\star$. 
Indeed, from the proof of \cref{lem:mass_laplace}, one obtains that the mass around the minimizer can be bounded as
\[
\delta_r \ge
(1-\delta)\frac{|B(0,r)|}{(2\pi)^{d/2}[\sigma(\sigma_0+3R)]^d}
\exp\!\left(-\frac{(r+C_R)^2}{2\sigma^2\sigma_0^2}\right) \gtrsim \exp\left(- \frac1{\sigma^2_0}\right)
\]
by the definition of constants $c_{2,\min}$ and $c_{2,\max}$.

Hence, in order that the Laplace remainder term satisfies
\[
\frac{e^{-\alpha q}}{\delta_r}\int |x-x^\star|\,\rho_{n+1}(dx)
\le \frac12\sqrt{\varepsilon\frac{\mu}{4C_1}},
\]
it is required to take
\[
\alpha \gtrsim  \frac1q\left( 
-\log(\veps)
+\frac1{\sigma_0^2}
\right)\,.
\]

Let us now  focus on the dependence with respect to the dimension $d$. Recall $|B(0,r)| = \pi^{d/2}r^d/ \Gamma(d/2 + 1)$, 
and that by the Stirling approximation it holds $\Gamma(d/2 + 1) \lesssim \sqrt{d}(\frac{d}{2e})^{d/2}$.
The estimate now becomes 
\begin{align*}
\delta_r \ge
(1-\delta)\frac{1}{2^{d/2}\Gamma(d/2 + 1)} \left(\frac{r}{\sigma(\sigma_0+3R)}\right)^d
\exp\!\left(-\frac{(r+C_R)^2}{2\sigma^2\sigma_0^2}\right)  \gtrsim \frac 1{\sqrt{d}}\left(\frac Cd  \right)^{d/2} 
\end{align*}
for some constant $C>0$. Similarly as above, this leads to a dimension-dependent constraint for $\alpha$ given by
\[
\alpha \gtrsim \frac1q \left( \frac{d}2 \log d - \frac d2 \log C \right)\,.
\]

\end{remark}

\begin{figure}[t]
    \centering
    \includegraphics[width=1\textwidth]{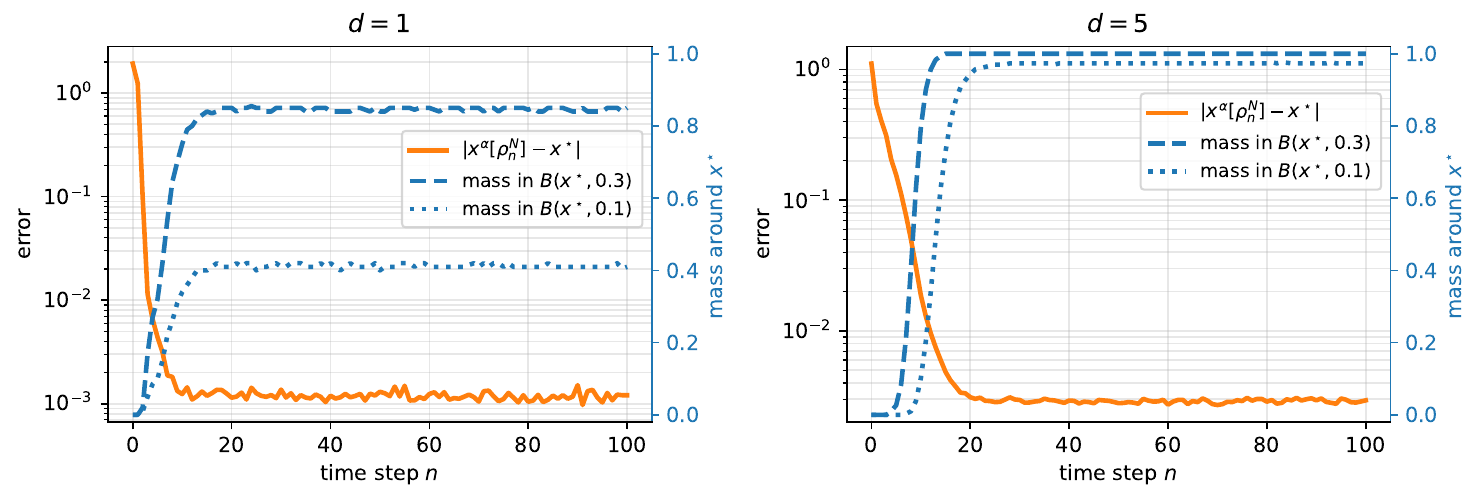}
    \caption{Mass concentration and consensus error for finite-particle simulations of
\eqref{eq:2nd-cbo-td-particle} on the Ackley benchmark. The two panels show
$d=1$ and $d=5$. All curves are pointwise medians over $200$ independent
runs. The orange curve, plotted on the left logarithmic axis, shows the
consensus error $|x^\alpha[\rho_n^N]-x^\star|$. The blue dashed and dotted
curves, plotted on the right axis, show the empirical mass around $x^\star$
for two values, $r = 0.1, 0.3$. The mass remains positive while the consensus error
decays toward a noise floor, as analyzed in the \cref{lem:mass_laplace}. Parameter values are listed in
\cref{app:parameters}.}
    \label{fig:mass-concentration}
\end{figure}

\begin{remark}
\Cref{fig:mass-concentration} illustrates the mass concentration mechanism analyzed in the proof of \cref{lem:mass_laplace}. A finite-particle simulation of
\eqref{eq:2nd-cbo-td-particle} on the Ackley benchmark shows that the
empirical mass near $x^\star$ saturates at a strictly positive level and
the consensus point $x^\alpha[\rho_n^N]$ is correspondingly pulled close
to $x^\star$, across both $d = 1$ and $d = 5$ and across two measurement
radii $r$.
\end{remark}

\subsection{Convergence towards minimizer }

From  \cref{lem:mass_laplace}, we observe that the consensus point remains close to the minimizer as long as the error functional stays bounded. Thanks to the contractivity of the dynamics established in  \cref{prop:h2_decay_dt1}, we also see that the error functional decays as long as the consensus point remains close to the minimizer. To prove convergence, we therefore bootstrap these two arguments to show that the error decays toward a prescribed accuracy.
We conclude the section with a discussion of how the parameters $\alpha$ and $\sigma_0$ depend on the accuracy $\veps$, see  \cref{rmk:constants2}.

  \begin{theorem} \label{t:minimizer}
Let $x^\star \in \supp(\rho_0)$ and $\veps>0$ be fixed. Assume that the conditions on the system parameters imposed on \cref{lem:mass_laplace} all hold. If $\calE$ satisfies \cref{asm:inverse}, there exists
$\alpha=\alpha(\rho_0, \veps)\gg 1$ and $\sigma_0=\sigma_0(\rho_0, \veps)\ll 1$ such that
\[
\err(\OX_n,\OV_n)\le \max\{(1-\mu/2)^n \err(\OX_0,\OV_0),\,\veps\}
\qquad \textup{for all}\quad  n\ge 0.
\]

As a consequence, it holds
\[
\err(\OX_{n}, \OV_{n}) \leq \veps \,\qquad \textup{for all}\quad n\geq n_T := \frac2{\mu} \log\left( \frac{\err(\OX_0, \OV_0)}{\veps}\right) \,.
\]
\end{theorem}
\begin{proof}
Since $x^\star \in \mathrm{supp}(\rho_0)$, for any radius $r>0$ it holds $\rho_{0}(B(x^\star, r))>0$. Therefore, at $n = 0$ we can directly apply the quantitative Laplace principle \cref{prop:laplace} and state that there exists $\alpha_0 = \alpha_0(\rho_0, \veps)>0$ sufficiently large such that $|x^\alpha[\rho_0] - x^\star|\leq \sqrt{\veps \mu/ (4C_1)}$, where $ C_1$ is the same constant appearing in \cref{prop:h2_decay_dt1} and \cref{lem:mass_laplace}.

Set $\sigma_0^2 := \veps \frac{\mu}{4C_1}$ and let  $E_0:=\max\{\err(\OX_0,\OV_0),\veps\}$. 
By Lemma \ref{lem:mass_laplace}, we can now choose $\alpha=\alpha(E_0,\veps)> \alpha_0$
 so that whenever $\err(\OX_n,\OV_n)\le E_0$ and
$|x^\alpha[\rho_n]-x^\star|^2\le \veps \mu/(4C_1)$, then also
\[
|x^\alpha[\rho_{n+1}]-x^\star|^2\le \veps\frac{\mu}{4C_1}.
\]
We prove \begin{equation}\label{eq:induction_pair}
    \err(\OX_n,\OV_n)\le E_0,
    \qquad
    |x^\alpha[\rho_n]-x^\star|^2\le \veps\frac{\mu}{4C_1}
\end{equation} 
by induction on $n$. The base case $n = 0$ holds by construction of $\sigma_0^2$ and $\alpha$. Assume now \eqref{eq:induction_pair} holds at time $n$. Then, by the choice of
$\alpha$, we immediately get
\[
|x^\alpha[\rho_{n+1}]-x^\star|^2\le \veps\frac{\mu}{4C_1}.
\]
Moreover, using the almost-contractivity of the error functional, \cref{prop:h2_decay_dt1} we have
\begin{align*}
    \err(\OX_{n+1}, \OV_{n+1})
    &\le (1-\mu)\,\err(\OX_n,\OV_n) + C_1\bigl(\sigma_0^2 + |x^\alpha[\rho_n]-x^\star|^2 \bigr) \\
    &\le (1-\mu)\,\err(\OX_n,\OV_n) + C_1\Bigl(\veps\frac{\mu}{4C_1}+\veps\frac{\mu}{4C_1}\Bigr) \\
    &= (1-\mu)\,\err(\OX_n,\OV_n) + \veps\frac{\mu}{2}.
\end{align*}
In particular, if $\err(\OX_n,\OV_n)\ge \veps$, then $\veps\le \err(\OX_n,\OV_n)$ and
\[
\err(\OX_{n+1},\OV_{n+1})
\le (1-\mu)\,\err(\OX_n,\OV_n) + \frac{\mu}{2}\,\err(\OX_n,\OV_n)
= (1-\mu/2)\,\err(\OX_n,\OV_n).
\]
If instead $\err(\OX_n,\OV_n)\le \veps$, then
\[
\err(\OX_{n+1},\OV_{n+1})
\le (1-\mu)\,\veps + \veps\frac{\mu}{2} \le \veps.
\]
Combining the two cases yields the one-step estimate
\[
\err(\OX_{n+1},\OV_{n+1})
\le \max\{(1-\mu/2)\err(\OX_n,\OV_n),\veps\}.
\]
In particular, $\err(\OX_{n+1},\OV_{n+1})\le E_0$, so \eqref{eq:induction_pair}
holds at time $n+1$. This closes the induction and implies
\begin{equation}\label{eq:err_final_bound}
\err(\OX_n,\OV_n)\le \max\{(1-\mu/2)^n\err(\OX_0,\OV_0),\,\veps\}\qquad \forall n\ge 0.
\end{equation}
The fact that $\err(\OX_{n}, \OV_{n}) \leq \veps$ for $n\geq n_T$ follows by applying the definition of $n_T$ in the above bound.

\end{proof}

\begin{remark}\label{rmk:constants2} In view of the choice $\sigma_0^2$ in the proof of the theorem and the relation between $\alpha$ and $\sigma_0$ outlined in  \cref{rmk:constants1}, we have that algorithm's constants depend on the tolerance $\veps$ as 
\[
\sigma_0 \lesssim \sqrt{\veps}\qquad \textup{and}\qquad \alpha \gtrsim \frac1\veps \,.
\]
\end{remark}

\section{Mean-field approximation and finite-particle convergence}\label{sec: mfl}
\subsection{Mean-field limit}
To justify the mean-field approximation \eqref{eq:2nd-cbo-td-mf}, we employ a synchronous coupling argument from \cite{MR1108185}. However, establishing this coupling for this second-order model is challenged by the non-Lipschitz nature of the drift and the noise term, and finally the multiplicative noise. Such a challenge is not a unqiue one for PSO, but something CBO variants all have in common. Hence we overcome such a challenge by employing a similar approach from \cite{gerber2025mean, Hui_Huang_2025_cpaa}.

More recently, there were some uniform-in-time mean-field limit results covering  both first-order and second-order CBO where the latter is a continuous time analogue of the model that we study here \cite{gerber2026uniformintimepropagationchaosconsensusbased, ha2026uniformintimepropagationchaossecondorder}. Previously, there was already a finite-time mean field limit result for the same continuous time, second-order model \cite{huang_2021_note}. The recent work \cite{ha2026uniformintimepropagationchaossecondorder} shows that under higher assumption and high friction, low inertia regime, this result holds for an infinite time horizon. 

The architecture of our proof is as follows: First, we establish uniform moment bounds for both the discrete particle system (\cref{lem: pathwise-moment-bound-particle}) and the mean-field process (\cref{lem:pathwise-moment-bound-time-discrete-mf}). These bounds allow us to confine the dynamics to a high-probability compact set where the drift coefficients are locally Lipschitz. We then bound the error of the empirical consensus point using a law-of-large-numbers argument (\cref{lem:law of large number}), which ultimately allows us to close the discrete Grönwall estimate in \cref{thm: mfl-2nd-order-cbo}.

\begin{assumption}\label{asm: E-bound}
    We assume that the objective function $\mathcal{E}$ satisfies the following: 
\begin{enumerate}
    \item For all $x, y \in \R^d\,,$ 
    \begin{align}\label{eq: E-lip}
        |\, \mathcal{E}(x) - \mathcal{E}(y)\,| \leq L_{\mathcal{E}} ( 1 + |x| + |y|) |x - y|
    \end{align}
    for some constant $L_{\mathcal{E}} > 0$.
    \item There exist constants $c_l, C_l, c_u,$ and $ C_u > 0$ such that for all $x \in \R^d$,
    \begin{align}
        \mathcal{E}(x) - \mathcal{E}(x^*) &\leq c_u |x|^2 + C_u\,, \label{eq: E-bound-upper} \\
        \mathcal{E}(x) -  \mathcal{E}(x^*) &\geq c_l |x|^2 - C_l\,. \label{eq: E-bound-lower}
    \end{align}
\end{enumerate}
\end{assumption}

First, let us recall some stability estimates for the consensus point $x^\alpha(\cdot)$:
\begin{lemma}{\cite[Corollary 3.3]{gerber2025mean}}\label{lem:stability}
Suppose that $\mathcal{E}$ satisfies \cref{asm: E-bound}. Then for all $R > 0\,,$ there exists $L_{M} = L_M(R) > 0$ such that 
   \begin{equation}
       \forall~(\mu,\nu)\in \cP_{2,R}(\R^d)\times \cP_2(\R^d),\quad
    |x^\alpha(\mu)-x^\alpha(\nu)|\leq L_M W_2(\mu,\nu) \,,
   \end{equation} 
and for all $p\geq 1$, there exists a constant $C_1$ depending on $p, c_u, C_u, c_l, C_l$ such that
\begin{equation}\label{bound of xa}
    \forall \mu\in \mathcal{P}_{p}(\R^d),\quad |x^\alpha(\mu)|^p\leq C_1\left(1+\int_{\R^d}|x|^p\mu(dx)\right)\,.
\end{equation}
\end{lemma}

We shall also establish the following pathwise moment bound for the particle system.

\begin{lemma}\label{lem: pathwise-moment-bound-particle}
    Let $\mathcal{E}$ satisfy \cref{asm: E-bound} and suppose the initial data satisfies $f_0 \in \mathcal{P}_p(\R^{2d})$ for any fixed $p \geq 2$. Consider the discrete particle system \eqref{eq:2nd-cbo-td-particle}. Then, for any fixed maximum number of steps $n_T > 0$, there exists a constant $\kappa > 0$, independent of $N$, such that
    \begin{align}
        \E\left[ \sup_{0 \leq n \leq n_T} |X_n^i|^p \right] 
        \vee \E\left[ \sup_{0 \leq n \leq n_T} |V_n^i|^p \right] 
        \vee \E\left[ \sup_{0 \leq n \leq n_T} | x^{\alpha}[\rho_n^N] |^p \right] \leq \kappa\,,
    \end{align}
    where $\rho_n^N$ is the empirical measure associated with the particle positions $\{X_n^i\}_{i=1}^N$.
\end{lemma}

\begin{proof}
    We proceed by unrolling the discrete scheme into a summation from $j=0$ to $k-1$. For the velocity update, we have:
    \begin{equation}
        V_k^i = V_0^i + \sum_{j=0}^{k-1} \Bigl[ -\gamma V_j^i + \lambda\bigl(x^\alpha[\rho_j^N] - X_j^i\bigr) \Bigr] + \sum_{j=0}^{k-1} \sigma\bigl(\sigma_0 + | x^\alpha[\rho_j^N] - X_j^i |\bigr)\xi_j^i\,.
    \end{equation}
    Using the inequality $(a+b+c)^p \leq 3^{p-1}(a^p + b^p + c^p)$ and applying Hölder's inequality to the deterministic sum, we bound $|V_k^i|^p$:
    \begin{align}\label{eq: unrolled-v}
        |V_k^i|^p &\leq 3^{p-1} |V_0^i|^p + 3^{p-1} k^{p-1} \sum_{j=0}^{k-1} \left| -\gamma V_j^i + \lambda\bigl(x^\alpha[\rho_j^N] - X_j^i\bigr) \right|^p \notag \\
        &\quad + 3^{p-1} \left| \sum_{j=0}^{k-1} \sigma\bigl(\sigma_0 + | x^\alpha[\rho_j^N] - X_j^i |\bigr)\xi_j^i \right|^p\,.
    \end{align}
    We now take the supremum over all $k \in \{0, \dots, n\}$ (where $n \leq n_T$) and bound $k \leq n$:
    \begin{equation}
        \sup_{0 \leq k \leq n} |V_k^i|^p \leq 3^{p-1} |V_0^i|^p + C_p n^{p-1} \sum_{j=0}^{n-1} \Bigl( |V_j^i|^p + |X_j^i|^p + |x^\alpha[\rho_j^N]|^p \Bigr) + 3^{p-1} \sup_{0 \leq k \leq n} |M_k|^p\,,
    \end{equation}
    where $M_k := \sum_{j=0}^{k-1} \sigma_j \xi_j^i$ with $\sigma_j := \sigma(\sigma_0 + | x^\alpha[\rho_j^N] - X_j^i |)$ is a discrete martingale with respect to the natural filtration $\mathcal{F}_k$ generated by the particle history, since $\E[\xi_j^i | \mathcal{F}_j] = 0$.
    
    We take the expectation of both sides. To bound the martingale term, we apply the discrete Burkholder-Davis-Gundy (BDG) inequality \cite[Chapter VII]{shiryaev1995probability}:
    \begin{equation}
        \E\left[ \sup_{0 \leq k \leq n} |M_k|^p \right] \leq C_p^{BDG} \E\left[ \left( \sum_{j=0}^{n-1} \sigma_j^2 |\xi_j^i|^2 \right)^{p/2} \right]\,.
    \end{equation}
    Applying Hölder's inequality to the sum inside the expectation gives $(\sum_{j=0}^{n-1} A_j)^{p/2} \leq n^{p/2-1} \sum_{j=0}^{n-1} A_j^{p/2}$. Using the independence of $\xi_j^i$ and the fact that $\E[|\xi_j^i|^p] < \infty$, we obtain:
    \begin{align}
        \E\left[ \sup_{0 \leq k \leq n} |M_k|^p \right] &\leq \widetilde{C}_p n^{p/2-1} \sum_{j=0}^{n-1} \E\left[ |\sigma_j|^p \right] \notag \\
        &\leq \widehat{C}_p n^{p/2-1} \sum_{j=0}^{n-1} \E\left[ 1 + |X_j^i|^p + |x^\alpha[\rho_j^N]|^p \right]\,.
    \end{align}
    
    Similarly, unrolling the position update $X_k^i = X_0^i + \sum_{j=0}^{k-1} V_{j+1}^i$, taking the $p$-th power, applying Hölder's inequality, taking the supremum, and taking the expectation yields:
    \begin{equation}\label{eq: unrolled-x}
        \E\left[ \sup_{0 \leq k \leq n} |X_k^i|^p \right] \leq 2^{p-1} \E|X_0^i|^p + 2^{p-1} n^{p-1} \sum_{j=0}^{n-1} \E\left[ \sup_{0 \leq m \leq j+1} |V_m^i|^p \right]\,.
    \end{equation}
    
    Now, let $E_n := \E\left[ \sup_{0 \leq k \leq n} |X_k^i|^p \right]$ and $F_n := \E\left[ \sup_{0 \leq k \leq n} |V_k^i|^p \right]$. Using the  bound \eqref{bound of xa}, we can bound the empirical consensus point:
    \begin{equation}
        \E\left[ \sup_{0 \leq j \leq n} |x^\alpha[\rho_j^N]|^p \right] \leq C_1 \left( 1 + \frac{1}{N}\sum_{l=1}^N \E\left[ \sup_{0 \leq j \leq n} |X_j^l|^p \right] \right) = C_1(1 + E_n)\,.
    \end{equation}
    
    Combining all the estimates into our bounds for $E_n$ and $F_n$, and grouping terms dependent on $n$ into a single constant $C(n_T)$ (since $n \leq n_T$), we arrive at a coupled system of inequalities:
    \begin{align}
        F_n &\leq C(n_T) \left( 1 + \E|V_0^i|^p + \sum_{j=0}^{n-1} \left( E_j + F_j \right) \right)\,, \\
        E_n &\leq C(n_T) \left( 1 + \E|X_0^i|^p + \sum_{j=0}^{n} F_j \right)\,.
    \end{align}
    Let $S_n := E_n + F_n$. Substituting $F_n$ into the bound for $E_n$ and re-arranging reveals that $S_n$ satisfies a discrete Grönwall inequality of the form:
    \begin{equation}
        S_n \leq C_A + C_B \sum_{j=0}^{n-1} S_j\,,
    \end{equation}
    where $C_A$ depends on the initial moments $\E|X_0^i|^p + \E|V_0^i|^p$, which are finite by assumption and independent of $N$. By the discrete Grönwall lemma, we conclude:
    \begin{equation}
        S_n \leq C_A \exp(C_B n) \leq C_A \exp(C_B n_T) =: \kappa\,.
    \end{equation}
    Because this holds for $S_{n_T} = E_{n_T} + F_{n_T}$, both $\E[\sup_{0 \leq k \leq n_T} |X_k^i|^p]$ and $\E[\sup_{0 \leq k \leq n_T} |V_k^i|^p]$ are bounded by $\kappa$. The bound on the consensus point follows immediately.
\end{proof}

Similarly, we establish the corresponding moment bounds for the solution to the mean-field equation at the discrete level.

\begin{lemma}\label{lem:pathwise-moment-bound-time-discrete-mf}
    Let $\mathcal{E}$ satisfy \cref{asm: E-bound} and suppose $f_0 \in \mathcal{P}_p(\R^{2d})$ for any  $p \geq 2$. Consider the mean-field \eqref{eq:2nd-cbo-td-mf} with initial data distributed according to $f_0$. Then, for any fixed maximum number of steps $n_T > 0$, there exists a constant $\kappa > 0$
    \begin{align}
        \E\left[ \sup_{0 \leq n \leq n_T} |\OX_n|^p \right] 
        \vee \E\left[ \sup_{0 \leq n \leq n_T} |\OV_n|^p \right] 
        \vee  \sup_{0 \leq n \leq n_T} | x^{\alpha}[\rho_n] |^p  \leq \kappa\,,
    \end{align}
    where $\rho_n$ is the law of the position $\OX_n$ at step $n$.
\end{lemma}

We further collect some results from \cite{gerber2025mean}. The first is a bound on the probability of large excursions:
\begin{lemma}{\cite[Lemma 2.5]{gerber2025mean}}\label{lem:bound on probability}
    Let $\{Z_i\}_{i=1}^N$ be a family of i.i.d. $\R$-valued random variables such that $\E[|Z_1|^r]<\infty$ for some $r\geq 2$. Then for all $R>\E[|Z_1|]$, there exists a constant $C>0$ such that
    \begin{equation}
        \proba\left(\frac{1}{N}\sum_{i=1}^NZ_i\geq R\right)\leq CN^{-r/2}\,.
    \end{equation}
\end{lemma}
The second ensures the convergence of the weighted mean for i.i.d. samples:
\begin{lemma}{\cite[Lemma 3.7]{gerber2025mean}}\label{lem:law of large number}
Assume that $\mathcal{E}$ satisfies \cref{asm: E-bound}. Let $0 < p < r.$ For all $\rho_n \in \mathcal{P}_r(\R^d)\,,$ there is a constant $C_1:= C_1( \mathcal{E}, \alpha, p, r, \|x\|_{L^r(\mu)})$ such that for all $N \in \mathbb{N}\,,$
\begin{equation}
    \E\left[\, \sup_{n\in \mathbb{N}} |x^\alpha[\overline{\rho}_n^N ]-x^\alpha[\rho_n]|^p\,\right]\leq C_1 N^{-p/2} \,,
\end{equation}
    where $\overline{\rho}_n^N  = \frac{1}{N} \sum_{i = 1}^N \delta_{\OX^i}$ for $ \left\{ \OX^i \right\}_{i \in \mathbb{N}}  \stackrel {\text{i.i.d.}} \sim \mu$.
\end{lemma}

\begin{theorem}\label{thm: mfl-2nd-order-cbo}
Suppose that $\mathcal{E}$ satisfies \cref{asm: E-bound} and $f_0\in \mathcal{P}_6(\R^{2d})$. Consider the systems \eqref{eq:2nd-cbo-td-particle} and \eqref{eq:2nd-cbo-td-mf} with initial conditions $X_0^i = \OX_0^i$ and $V_0^i = \OV_0^i$. Then for any fixed maximum number of steps $n_T > 0$, there exists a constant $C(n_T)>0$ independent of $N$ such that
    \begin{equation}
        \sup_{i=1,\dots,N} \E\left[ \sup_{0 \leq n \leq n_T} \left( |X_{n}^i-\OX_{n}^i| + |V_{n}^i-\OV_{n}^i| \right) \right] \leq C(n_T) N^{-1/2}\,.
    \end{equation}
\end{theorem}

\begin{proof}
    Fixing a particle index $i$, we begin by unrolling the velocity update from $k=0$ to $k=n-1$. Let $\Delta b_j := b(X_j^i, \rho_j^N) - b(\OX_j^i, \rho_j)$. The difference in velocity updates can be written recursively as $V_{j+1}^i - \OV_{j+1}^i = (1 - \gamma)(V_j^i - \OV_j^i) + \lambda \Delta b_j + \sigma \Delta b_j \xi_j^i$. Because the initial conditions are synchronously coupled, $X_0^i - \OX_0^i = 0$ and $V_0^i - \OV_0^i = 0$, summing these differences yields:
    \begin{equation}
        V_n^i - \OV_n^i = - \gamma \sum_{j=0}^{n-1} (V_j^i - \OV_j^i) + \lambda \sum_{j=0}^{n-1} \Delta b_j + \sigma \sum_{j=0}^{n-1} \Delta b_j \xi_j^i\,.
    \end{equation}
    Using the inequality $(a+b+c)^2 \leq 3a^2 + 3b^2 + 3c^2$ and applying the Cauchy-Schwarz inequality to the first two deterministic sums, we bound the squared velocity difference at any step $k \leq n$:
    \begin{equation}
        |V_k^i - \OV_k^i|^2 \leq 3 \gamma^2 k \sum_{j=0}^{k-1} |V_j^i - \OV_j^i|^2 + 3 \lambda^2 k \sum_{j=0}^{k-1} |\Delta b_j|^2 + 3 \left| \sum_{j=0}^{k-1} \sigma \Delta b_j \xi_j^i \right|^2\,.
    \end{equation}
    Taking the supremum over $k \in \{0, \dots, n\}$ and bounding $k \leq n$:
    \begin{align}\label{eq: vel-sup-bound}
        \sup_{0 \leq k \leq n} |V_k^i - \OV_k^i|^2 &\leq 3 \gamma^2 n \sum_{j=0}^{n-1} |V_j^i - \OV_j^i|^2 + 3 \lambda^2 n \sum_{j=0}^{n-1} |\Delta b_j|^2 + 3 \sup_{0 \leq k \leq n} \left| M_k \right|^2\,,
    \end{align}
    where $M_k := \sum_{j=0}^{k-1} \sigma \Delta b_j \xi_j^i$ is a discrete martingale with respect to the filtration $\mathcal{F}_k$ generated by the particle history, since $\E[\xi_j^i | \mathcal{F}_j] = 0$.
    
    Taking the expectation of \eqref{eq: vel-sup-bound}, we apply the discrete BDG inequality:
    \begin{equation}
        \E\left[ \sup_{0 \leq k \leq n} |M_k|^2 \right] \leq C_2^{BDG} \E[|M_n|^2] = C_2^{BDG} d \sigma^2 \sum_{j=0}^{n-1} \E[|\Delta b_j|^2]\,,
    \end{equation}
    where we used the Itô isometry equivalent for discrete martingales with $\E[|\xi_j^i|^2] = d$.
    
    Substituting this back, we obtain a pathwise bound for the velocity:
    \begin{equation}\label{eq: vel-pathwise-expected}
        \E\left[ \sup_{0 \leq k \leq n} |V_k^i - \OV_k^i|^2 \right] \leq 3 \gamma^2 n \sum_{j=0}^{n-1} \E\left[ \sup_{0 \leq m \leq j} |V_m^i - \OV_m^i|^2 \right] + (3 \lambda^2 n + 12 d \sigma^2) \sum_{j=0}^{n-1} \E[|\Delta b_j|^2]\,.
    \end{equation}
    
    For the position update, unrolling $X_k^i - \OX_k^i = \sum_{j=0}^{k-1} (V_{j+1}^i - \OV_{j+1}^i)$ and applying Cauchy-Schwarz yields:
    \begin{equation}\label{eq: pos-pathwise-expected}
        \E\left[ \sup_{0 \leq k \leq n} |X_k^i - \OX_k^i|^2 \right] \leq n \sum_{j=0}^{n-1} \E|V_{j+1}^i - \OV_{j+1}^i|^2 \leq n^2 \E\left[ \sup_{0 \leq k \leq n} |V_k^i - \OV_k^i|^2 \right]\,.
    \end{equation}
    
    Now we must bound the drift difference $\Delta b_j = (X_j^i - \OX_j^i) - \bigl( x^{\alpha}[\rho_j^N] - x^{\alpha}[\rho_j] \bigr)$. By the triangle inequality:
    \begin{align}\label{eq: delta-b-bound}
        \E[|\Delta b_j|^2] &\leq 2 \E\left[ \sup_{0 \leq m \leq j} |X_m^i - \OX_m^i|^2 \right] + 4\E|x^{\alpha}[\rho_j^N] - x^{\alpha}[\overline{\rho}_j^N]|^2 + 4\E|x^{\alpha}[\overline{\rho}_j^N] - x^{\alpha}[\rho_j]|^2 \notag \\
        &\leq 2 \E\left[ \sup_{0 \leq m \leq j} |X_m^i - \OX_m^i|^2 \right] + 4\E|x^{\alpha}[\rho_j^N] - x^{\alpha}[\overline{\rho}_j^N]|^2 + 4 C_1 N^{-1}\,,
    \end{align}
    where the last term is bounded by \cref{lem:law of large number}.

    To handle the non-globally Lipschitz weighted mean, we define the global excursion set $\Omega_N^R$ over the entire evaluated path:
    \begin{equation*}
        \Omega_N^R := \left\{\omega\in\Omega :~ \frac{1}{N}\sum_{l=1}^N \sup_{0 \leq k \leq n_T} |\OX_k^l(\omega)|^2 \geq R\right\}\,.
    \end{equation*}
    By \cref{lem:pathwise-moment-bound-time-discrete-mf}, $\E[ \sup_{k} |\OX_k^l|^{6} ] < \infty$. Thus, applying \cref{lem:bound on probability} with $r=3$, we find $\proba(\Omega_N^R) \leq C N^{-3/2}$.

    Splitting the expectation using indicator functions:
    \begin{align*}
        \E|x^\alpha(\rho_j^N)-x^\alpha(\overline{\rho}_j^N)|^2 &= \E\left[|x^\alpha(\rho_j^N)-x^\alpha(\overline{\rho}_j^N)|^2 \mathbf{1}_{\Omega \setminus \Omega_N^R}\right] + \E\left[|x^\alpha(\rho_j^N)-x^\alpha(\overline{\rho}_j^N)|^2 \mathbf{1}_{\Omega_N^R}\right]\,.  
    \end{align*}
    On the set $\Omega \setminus \Omega_N^R$, the empirical measure is bounded by $R$ for all steps up to $n_T$, allowing us to use \cref{lem:stability} uniformly:
    \begin{align*}
       \E\left[|x^\alpha(\rho_j^N)-x^\alpha(\overline{\rho}_j^N)|^2 \mathbf{1}_{\Omega \setminus \Omega_N^R}\right] &\leq L_M^2 \E[\mathcal{W}_2^2(\rho_j^N,\overline{\rho}_j^N)] \leq L_M^2 \E\left[ \sup_{0 \leq m \leq j} |X_m^i-\OX_m^i|^2 \right]\,.
    \end{align*}
    For the complementary set, Hölder's inequality and the uniform $p=6$ moment bounds yield:
    \begin{align*}
       \E\left[|x^\alpha(\rho_j^N)-x^\alpha(\overline{\rho}_j^N)|^2 \mathbf{1}_{\Omega_N^R}\right] &\leq \E\left[|x^\alpha(\rho_j^N)-x^\alpha(\overline{\rho}_j^N)|^6 \mathbf{1}_{\Omega_N^R}\right]^{1/3} \bigl(\proba(\Omega_N^R)\bigr)^{2/3} \leq C (N^{-3/2})^{2/3} = C N^{-1}\,.
    \end{align*}
    Substituting this back into \eqref{eq: delta-b-bound}, we get:
    \begin{equation}\label{eq: delta-b-final}
        \E[|\Delta b_j|^2] \leq C \E\left[ \sup_{0 \leq m \leq j} |X_m^i - \OX_m^i|^2 \right] + C N^{-1}\,.
    \end{equation}

    Let $E_n := \E\left[ \sup_{0 \leq k \leq n} |X_k^i - \OX_k^i|^2 \right]$ and $F_n := \E\left[ \sup_{0 \leq k \leq n} |V_k^i - \OV_k^i|^2 \right]$. Substituting \eqref{eq: delta-b-final} into \eqref{eq: vel-pathwise-expected}, and utilizing the fact that $E_j \leq j^2 F_j \leq n^2 F_j$ from \eqref{eq: pos-pathwise-expected}, we obtain:
    \begin{align}
        F_n &\leq C n \sum_{j=0}^{n-1} F_j + C n \sum_{j=0}^{n-1} \left( E_j + N^{-1} \right) \notag \\
        &\leq \widetilde{C}(n) \sum_{j=0}^{n-1} F_j + \widetilde{C}(n) N^{-1}\,,
    \end{align}
    where $\widetilde{C}(n)$ is a polynomial in $n$ derived from the constants. Because $n \leq n_T$, we can bound $\widetilde{C}(n) \leq \widetilde{C}(n_T)$.
    
    Applying the discrete Grönwall inequality yields $F_{n_T} \leq C(n_T) N^{-1}$. Consequently, $E_{n_T} \leq n_T^2 F_{n_T} \leq C(n_T) N^{-1}$. 
    
    Finally, using Jensen's inequality:
    \begin{align*}
         \E\left[ \sup_{0 \leq k \leq n_T} \left( |X_{k}^i-\OX_{k}^i| + |V_{k}^i-\OV_{k}^i| \right) \right] &\leq \left( 2 E_{n_T} + 2 F_{n_T} \right)^{1/2} \leq \sqrt{C(n_T)} N^{-1/2}\,.
    \end{align*}
\end{proof}

\subsection{Total error estimate for the numerical scheme}
Now collecting results from \cref{t:minimizer} and \cref{thm: mfl-2nd-order-cbo}
we can establish a quantitative convergence result for the numerical scheme \eqref{eq:2nd-cbo-td-particle}:
\begin{theorem}\label{thmmain}
Under the assumptions of \cref{t:minimizer} and \cref{thm: mfl-2nd-order-cbo}, let $\{(X_{n}^i, V_n^i)_{n=0,\dots,n_T}\}_{i=1}^N$ be the iterations generated by the particle system \eqref{eq:2nd-cbo-td-particle}, where $n_T$ comes from \cref{t:minimizer} such that $\err(\overline{X}_{n_T}, \overline{V}_{n_T}) \leq \varepsilon$ for any prescribed accuracy $\varepsilon>0$. Then the final iterations fulfill the following quantitative error estimate 
\begin{align}
	\expect\left[\left|\frac{1}{N}\sum_{i=1}^N X_{n_T}^i - x^*\right|^2\right] \leq 2C_{\mathrm{MFA}}\frac{1}{N} + \frac{2}{\lambda_{\min}(P)}\varepsilon\,,
\end{align}
where $C_{\mathrm{MFA}} > 0$ depends on $n_T$ as derived in the proof of \cref{thm: mfl-2nd-order-cbo}, and $\lambda_{\min}(P)$ is the minimum eigenvalue of the Lyapunov matrix $P$ defined in \cref{sec: contractivity}.
\end{theorem}

\begin{proof}
	Recall that $\{(X_n^i)_{n=0,\dots,n_T}\}_{i=1}^N$ and $\{(\overline{X}_n^i)_{n=0,\dots,n_T}\}_{i=1}^N$ are the position components of the $N$-particle system \eqref{eq:2nd-cbo-td-particle} and $N$ independent copies of the mean-field dynamics \eqref{eq:2nd-cbo-td-mf} up to step $n_T$, respectively. We split the squared error as follows:
	\begin{align}\label{eqerror}
		\expect\left[\left|\frac{1}{N}\sum_{i=1}^N X_{n_T}^i - x^*\right|^2\right]
		\leq & \, 2 \expect\left[\left|\frac{1}{N}\sum_{i=1}^N(X_{n_T}^i - \overline{X}_{n_T}^i)\right|^2\right] + 2 \expect\left[\left|\frac{1}{N}\sum_{i=1}^N \overline{X}_{n_T}^i - x^*\right|^2\right]\,,
	\end{align}
	which divides the overall error into the mean-field approximation error and the optimization error of the mean-field law.
	
	The first term on the right-hand side of \eqref{eqerror} can be bounded by applying Jensen's inequality and the mean-field coupling estimates in  \cref{thm: mfl-2nd-order-cbo}, which yields
	\begin{equation*}
		\expect\left[\left|\frac{1}{N}\sum_{i=1}^N(X_{n_T}^i - \overline{X}_{n_T}^i)\right|^2\right] \leq \frac{1}{N}\sum_{i=1}^N \expect\left[\left|X_{n_T}^i - \overline{X}_{n_T}^i\right|^2\right] \leq C_{\mathrm{MFA}}\frac{1}{N}\,.
	\end{equation*}
	Finally, the second term follows from the global convergence bound of \cref{t:minimizer}. Applying Jensen's inequality and the bounding property of the tilted error functional $\err(X, V)$ constructed via Stein's theorem, it holds that
	\begin{equation*}
		\expect\left[\left|\frac{1}{N}\sum_{i=1}^N \overline{X}_{n_T}^i - x^*\right|^2\right] \leq \frac{1}{N}\sum_{i=1}^N \expect\left[\left|\overline{X}_{n_T}^i - x^*\right|^2\right] = \expect\left[\left|\overline{X}_{n_T}^1 - x^*\right|^2\right] \leq \frac{1}{\lambda_{\min}(P)}\err(\overline{X}_{n_T}, \overline{V}_{n_T}) \leq \frac{\varepsilon}{\lambda_{\min}(P)}\,.
	\end{equation*}
	Combining the estimates above completes the proof.
\end{proof}

\begin{corollary}[Convergence of the empirical consensus point]
\label{cor:empirical-consensus-point}
Assume the hypotheses of \cref{t:minimizer,thm: mfl-2nd-order-cbo}, and let $n_T$ be chosen as in \cref{t:minimizer} so that $  \err(\overline X_{n_T},\overline V_{n_T})\le \varepsilon\,.$ Then there exists a constant
$\widehat{C}_{\mathrm{MFL}}(n_T)>0$, independent of $N$, such that
\begin{equation}
\label{eq:empirical-consensus-point-error}
    \expect\left[
    \left|x^\alpha[\rho_{n_T}^N]-x^\star\right|^2
    \right] \le
\frac{\widehat{C}_{\mathrm{MFL}}(n_T)}{N}
    +
\frac{3\mu}{4C_1} \varepsilon\,,
\end{equation}
where $\mu$ and \(C_1\) are the constants from  \cref{prop:h2_decay_dt1}. In particular, the weighted consensus point, as an output of the particle algorithm, converges to the global minimizer up to the finite-particle error and the prescribed optimization tolerance.
\end{corollary}

\begin{proof}
Let $\overline\rho_{n_T}^N
:=
\frac1N\sum_{i=1}^N\delta_{\overline X_{n_T}^i}\,,$ where
\((\overline X_n^i,\overline V_n^i)_{i=1}^N \) are independent copies of the mean-field dynamics \eqref{eq:2nd-cbo-td-mf}, synchronously coupled with the particle system. We decompose
\begin{align}
    \left|x^\alpha[\rho_{n_T}^N]-x^\star\right|^2 &\le
    3\left|
    x^\alpha[\rho_{n_T}^N] -  x^\alpha[\overline\rho_{n_T}^N] \right|^2 + 3\left| x^\alpha[\overline\rho_{n_T}^N] -x^\alpha[\rho_{n_T}] \right|^2 + 3\left| x^\alpha[\rho_{n_T}] - x^\star \right|^2.
\label{eq:consensus-error-split}
\end{align}

We first estimate the difference between the interacting empirical
consensus point and the empirical consensus point generated by the
independent mean-field copies. The same logic involving localization (via a set $\Omega_N^R$ of large excursion) and stability argument from the proof of \cref{thm: mfl-2nd-order-cbo} gives
\begin{equation}
\label{eq:consensus-coupling-bound}
    \expect\left[ \left| x^\alpha[\rho_{n_T}^N]
- x^\alpha[\overline\rho_{n_T}^N] \right|^2 \right] \le \frac{C_{\mathrm{coup}}(n_T)}{N}
\end{equation}
for some constant \(C_{\mathrm{coup}}(n_T)>0\) independent of \(N\). Next, since \(\overline X_{n_T}^1,\dots,\overline X_{n_T}^N\) are i.i.d. with law \(\rho_{n_T}\), the law-of-large-numbers estimate \cref{lem:law of large number} gives
\begin{equation}
\label{eq:consensus-lln-bound}
    \expect\left[ \left| x^\alpha[\overline\rho_{n_T}^N] - x^\alpha[\rho_{n_T}] \right|^2 \right] \le \frac{C_{\mathrm{LLN}}(n_T)}{N}
\end{equation}
for some constant \(C_{\mathrm{LLN}}(n_T)>0\) independent of \(N\). Finally, the bootstrap estimate established in the proof of
\cref{t:minimizer} implies
\begin{equation}
\label{eq:mean-field-consensus-bootstrap-bound}
    \left| x^\alpha[\rho_{n_T}] -x^\star \right|^2 \le \varepsilon\frac{\mu}{4C_1}.
\end{equation}

Taking expectations in \eqref{eq:consensus-error-split} and using
\eqref{eq:consensus-coupling-bound},
\eqref{eq:consensus-lln-bound}, and
\eqref{eq:mean-field-consensus-bootstrap-bound}, we obtain \[
    \expect\left[ \left|x^\alpha[\rho_{n_T}^N]-x^\star\right|^2 \right] \le \frac{ 3C_{\mathrm{coup}}(n_T) + 3C_{\mathrm{LLN}}(n_T)}{N} + \frac{3\mu}{4C_1} \varepsilon\,.
\]
Setting
$ \widehat{C}_{\mathrm{MFL}}(n_T) := 3C_{\mathrm{coup}}(n_T) + 3C_{\mathrm{LLN}}(n_T),$ we get \eqref{eq:empirical-consensus-point-error} as claimed.
\end{proof}

\section*{Acknowledgments}
GB was supported by the Wolfson Fellowship of the Royal Society “Uncertainty quantification, data-driven simulations and learning of multiscale complex systems governed by PDEs” of Prof. L. Pareschi at Heriot-Watt University. HH was partially supported by the Start-up grant from Hunan University. DK is supported by NSF CAREER Award 2340762 by Prof. F. Hoffmann at California Institute of Technology.

\printbibliography

\appendix

\section{Proof of auxiliary lemmas}

\begin{proof}[Proof of \cref{prop: h2-decay-small-dt}]\label{pf: h2-small-dt}
    We introduce the error variables as before:
    \[
    e_n:=\overline X_n-x^\star,\qquad
    u_n:=\overline V_n-\frac{\lambda}{\gamma}(x^\star-\overline X_n)
    =\overline V_n+\frac{\lambda}{\gamma} e_n,\qquad
    Z_n:=\binom{e_n}{u_n}\in\R^{2d}.
    \]
    By construction, $\err(\overline X_n,\overline V_n)=\E\big[Z_n^\top(P\otimes I_d)Z_n\big]$. Let $m_n := x^\alpha[\rho_n]$. 
    A direct computation from the $\Delta t$-dependent update rules yields the compact vector recursion:
    \begin{equation}\label{eq:Zn-recursion-small-dt}
    Z_{n+1}=(A(\Delta t)\otimes I_d)Z_n + \lambda\Delta t\,(b_{\Delta t}\otimes I_d)(x^\alpha[\rho_n]-x^\star)
    +\sigma\sqrt{\Delta t}\,(\sigma_0+|x^\alpha[\rho_n]-\overline X_n|)\,(b_{\Delta t}\otimes I_d)\xi_n\,,
    \end{equation}
    where the structure vector is $b_{\Delta t} := \binom{\Delta t}{1 + \frac{\lambda}{\gamma}\Delta t}$.

    Because $\xi_n \in \R^d$ is independent of $\mathcal{F}_n$ with $\E[\xi_n|\mathcal{F}_n]=0$ and $\E[|\xi_n|^2|\mathcal{F}_n]=d$, the cross-terms involving the noise vanish. Expanding $\E\big[Z_{n+1}^\top(P\otimes I_d)Z_{n+1}\big]$ yields:
    \begin{align}
        \err(\overline X_{n+1},\overline V_{n+1})
        &=\E\big[Z_n^\top(A(\Delta t)^\top P A(\Delta t)\otimes I_d)Z_n\big]
        +\lambda^2(\Delta t)^2(b_{\Delta t}^\top P b_{\Delta t})\,|x^\alpha[\rho_n]-x^\star|^2
        \nonumber\\
        &\quad
        +2\lambda\Delta t\,\E\Big[\big\langle (b_{\Delta t}^\top P A(\Delta t)\otimes I_d)Z_n,\ x^\alpha[\rho_n]-x^\star\big\rangle\Big] \nonumber\\
        &\quad + d\,\sigma^2 \Delta t(b_{\Delta t}^\top P b_{\Delta t})\,\E\big[(\sigma_0+|x^\alpha[\rho_n]-\overline X_n|)^2\big].
        \label{eq:H2-split-small-dt}
    \end{align}

    \smallskip
    \noindent\textbf{Step 1: Linear part via Taylor Expansion.} 
    From \cref{lem: dissipation-expansion}, we have the asymptotic expansion $A(\Delta t)^\top P A(\Delta t) = P - \Delta t Q_0 + \Delta t^2 R(\Delta t)$. Because $Q_0 \succ 0$ by assumption, let $\nu_0 := \lambda_{\min}(Q_0) > 0$. For $\Delta t$ sufficiently small, the remainder term is dominated, yielding $A(\Delta t)^\top P A(\Delta t) \preceq P - \Delta t \frac{\nu_0}{2} I_2$.
    Using the Rayleigh quotient lower bound $|z|^2 \ge \frac{1}{\lambda_{\max}(P)} z^\top(P\otimes I_d)z$, we obtain:
    \begin{equation}\label{eq:lin_part_bound_small}
    \E\big[Z_n^\top(A(\Delta t)^\top P A(\Delta t)\otimes I_d)Z_n\big] \le \left(1 - \Delta t \frac{\nu_0}{2\lambda_{\max}(P)}\right) \err(\overline X_n,\overline V_n) =: (1 - \hat{\mu}\Delta t) \err(\overline X_n,\overline V_n),
    \end{equation}
    where the continuous-time base rate is defined as $\hat{\mu} := \nu_0 / (2\lambda_{\max}(P)) > 0$.

    \smallskip
    \noindent\textbf{Step 2: Cross term.}
    By Cauchy--Schwarz and Young's inequality with parameter $\varepsilon \Delta t \lambda_{\min}(P) > 0$:
    \begin{align}\label{eq:cross_bound_small}
    2\lambda\Delta t\,\E\Big[\big\langle (b_{\Delta t}^\top P A(\Delta t)\otimes I_d)Z_n,\ x^\alpha[\rho_n]-x^\star\big\rangle\Big]
    &\le \varepsilon\Delta t \lambda_{\min}(P) \E|Z_n|^2 + \frac{\lambda^2 \Delta t}{\varepsilon \lambda_{\min}(P)}\,\|A(\Delta t)^\top P b_{\Delta t}\|_2^2\,|x^\alpha[\rho_n]-x^\star|^2 \nonumber \\
    &\le \varepsilon\Delta t\,\err(\overline X_n,\overline V_n)+\frac{\lambda^2 \Delta t}{\varepsilon \lambda_{\min}(P)}\,\|A(\Delta t)^\top P b_{\Delta t}\|_2^2\,|x^\alpha[\rho_n]-x^\star|^2.
    \end{align}

    \smallskip
    \noindent\textbf{Step 3: Noise term.}
    Using $|x^\alpha[\rho_n]-\overline X_n| \le |x^\alpha[\rho_n]-x^\star|+|e_n|$ and the inequality $(x+y+z)^2 \le 3(x^2+y^2+z^2)$:
    \[
    (\sigma_0+|x^\alpha[\rho_n]-\overline X_n|)^2 \le 3\big(\sigma_0^2+|x^\alpha[\rho_n]-x^\star|^2+|e_n|^2\big).
    \]
    Since $\E|e_n|^2 \le \E|Z_n|^2 \le \frac{1}{\lambda_{\min}(P)}\err(\overline X_n,\overline V_n)$, we have:
    \begin{align}\label{eq:noise_bound_small}
    &d\,\sigma^2 \Delta t(b_{\Delta t}^\top P b_{\Delta t})\,\E\big[(\sigma_0+|x^\alpha[\rho_n]-\overline X_n|)^2\big] \nonumber \\
    &\qquad\le \Delta t \frac{3d\sigma^2(b_{\Delta t}^\top P b_{\Delta t})}{\lambda_{\min}(P)}\err(\overline X_n,\overline V_n) + 3d\sigma^2 \Delta t(b_{\Delta t}^\top P b_{\Delta t})\big(\sigma_0^2+|x^\alpha[\rho_n]-x^\star|^2\big).
    \end{align}
    
    \smallskip
    \noindent\textbf{Step 4: Final combination.}
    \eqref{eq:lin_part_bound_small}, \eqref{eq:cross_bound_small}, and \eqref{eq:noise_bound_small} into \eqref{eq:H2-split-small-dt}, the multiplier for $\err(\overline X_n,\overline V_n)$ becomes:
    \[
    1 - \Delta t \left( \hat{\mu} - \varepsilon - \frac{3d\sigma^2(b_{\Delta t}^\top P b_{\Delta t})}{\lambda_{\min}(P)} \right).
    \]
    We choose the Young's inequality parameter $\varepsilon := \hat{\mu}/4$. Then, by assuming the noise $\sigma$ is sufficiently small such that $\frac{3d\sigma^2(b_{\Delta t}^\top P b_{\Delta t})}{\lambda_{\min}(P)} \le \hat{\mu}/4$, the net bracket is $1 - \Delta t(\hat{\mu}/2)$. We set the final decay rate $\mu_0 := \hat{\mu}/2 > 0$.
    
    Finally, fix \(\Delta t_0>0\) sufficiently small so that the estimate for the linear part and the small-noise absorption above hold for every \(0<\Delta t\leq \Delta t_0\). After collecting the coefficients of \(\lvert x^\alpha[\rho_n]-x^\star\rvert^2\) and factoring out \(\Delta t\), define
\[
\begin{aligned} 
    K :=
    \sup_{0<\Delta t\leq \Delta t_0}
    \Bigg\{&
    \lambda^2\Delta t\, b_{\Delta t}^\top P b_{\Delta t}
    +
    \frac{\lambda^2}{\varepsilon\,\lambda_{\min}(P)}
    \left\|A(\Delta t)^\top P b_{\Delta t}\right\|_2^2
    +
    3d\sigma^2\, b_{\Delta t}^\top P b_{\Delta t}
    \Bigg\}\,.
\end{aligned}
\]
Since \(A(\Delta t)\) and \(b_{\Delta t}\) depend continuously on \(\Delta t\) and remain bounded on the compact interval \([0,\Delta t_0]\), one has \(K<\infty\). Thus, \(K\) is independent of the individual step size \(\Delta t\in(0,\Delta t_0]\). The coefficient of \(\sigma_0^2\) is bounded by the same constant. Therefore, for every \(0<\Delta t\leq \Delta t_0\), \eqref{eq:H2-lemma-small-dt-form} holds. 
\end{proof}

\begin{lemma}[Asymptotic Expansion of the Dissipation Matrix]\label{lem: dissipation-expansion}
    Let $P = \begin{pmatrix} 1 & \theta \\ \theta & C \end{pmatrix}$ and let $A(\Delta t)$ be the transition matrix defined in \cref{prop:h2_decay_dt1}. We define the time-dependent dissipation matrix $Q(\Delta t)$ via the relation:
    \begin{equation}\label{eq:Q_def}
        A(\Delta t)^T P A(\Delta t) = P - \Delta t Q(\Delta t).
    \end{equation}
    Then, the entries of $Q(\Delta t)$ are polynomials in $\Delta t$. In the limit $\Delta t \to 0$, the leading-order dissipation matrix $Q_0 := \lim_{\Delta t \to 0} Q(\Delta t)$ is given by:
    \begin{equation}\label{eq:Q0_explicit}
        Q_0 = \begin{pmatrix} 
        \frac{2\lambda}{\gamma} + \frac{2\theta\lambda^2}{\gamma^2} & \gamma\theta - 1 + \frac{C\lambda^2}{\gamma^2} \\
        \gamma\theta - 1 + \frac{C\lambda^2}{\gamma^2} & 2C\gamma - \frac{2C\lambda}{\gamma} - 2\theta
        \end{pmatrix}.
    \end{equation}
    Furthermore, for every $\lambda,\gamma>0$, the interval
    \[
        \frac{1}{\gamma^2-\lambda+\lambda^2/\gamma^2}
        <
        C
        <
        \frac{\gamma^4}{\lambda^3}
        +
        \frac{\gamma^2}{\lambda^2}
    \]
is nonempty. For every $C$ in this interval and
    \[
        \theta=\frac{1}{\gamma}-\frac{C\lambda^2}{\gamma^3},
    \]
one has both $P\succ0$ and $Q_0\succ0$.
\end{lemma}

\begin{proof}
    We compute the entries of the matrix $P' := A^T P A$ as Taylor expansions in $\Delta t$ to identify the terms of order $O(1)$ and $O(\Delta t)$. From the definition \eqref{eq:Q_def}, the entries of the limit matrix are given by the derivative $Q_{0,ij} = -\frac{d}{d(\Delta t)} (A^T P A)_{ij} \big|_{\Delta t=0}$.
    
    \paragraph{1. The First Diagonal Entry ($Q_{0,11}$).}
    Expanding the term $(A^T P A)_{11} = |A_{11}|^2 + C|A_{21}|^2 + 2\theta A_{11}A_{21}$:
    \begin{align*}
        (A^T P A)_{11} &= \left(1-\tfrac{\lambda}{\gamma}\Delta t\right)^2 + C\left(\tfrac{\lambda}{\gamma}\right)^4(\Delta t)^2 - 2\theta\left(1-\tfrac{\lambda}{\gamma}\Delta t\right)\left(\tfrac{\lambda}{\gamma}\right)^2\Delta t \\
        &= 1 - \tfrac{2\lambda}{\gamma}\Delta t + \tfrac{\lambda^2}{\gamma^2}(\Delta t)^2 + O(\Delta t^2) - 2\theta\tfrac{\lambda^2}{\gamma^2}\Delta t + O(\Delta t^2) \\
        &= 1 - \Delta t \left( \tfrac{2\lambda}{\gamma} + \tfrac{2\theta\lambda^2}{\gamma^2} \right) + O(\Delta t^2).
    \end{align*}
    Since $P_{11}=1$, we identify the dissipation rate:
    \[
    Q_{0,11} = \lim_{\Delta t \to 0} \frac{1 - (A^T P A)_{11}}{\Delta t} = \frac{2\lambda}{\gamma} + \frac{2\theta\lambda^2}{\gamma^2}.
    \]

    \paragraph{2. The Second Diagonal Entry ($Q_{0,22}$).}
    Expanding $(A^T P A)_{22} = |A_{12}|^2 + C|A_{22}|^2 + 2\theta A_{12}A_{22}$:
    \begin{align*}
        (A^T P A)_{22} &= (\Delta t)^2(1-\gamma\Delta t)^2 + C(1-\gamma\Delta t)^2\left(1+\tfrac{\lambda}{\gamma}\Delta t\right)^2 + 2\theta\Delta t(1-\gamma\Delta t)^2\left(1+\tfrac{\lambda}{\gamma}\Delta t\right).
    \end{align*}
    We retain only terms up to $O(\Delta t)$:
    \begin{align*}
        (A^T P A)_{22} &= 0 + C\left[ 1 + 2\Delta t\left(\tfrac{\lambda}{\gamma} - \gamma\right) \right] + 2\theta\Delta t + O(\Delta t^2) \\
        &= C - \Delta t \left( 2C\gamma - \tfrac{2C\lambda}{\gamma} - 2\theta \right) + O(\Delta t^2).
    \end{align*}
    Since $P_{22}=C$, we have:
    \[
    Q_{0,22} = \lim_{\Delta t \to 0} \frac{C - (A^T P A)_{22}}{\Delta t} = 2C\gamma - \frac{2C\lambda}{\gamma} - 2\theta.
    \]

    \paragraph{3. The Off-Diagonal Entry ($Q_{0,12}$).}
    Expanding the cross term $(A^T P A)_{12} = A_{11} A_{12} + C A_{21} A_{22} + \theta (A_{11} A_{22} + A_{12} A_{21})$:
    \begin{itemize}
        \item $A_{11}A_{12} = (1-\tfrac{\lambda}{\gamma}\Delta t)\Delta t(1-\gamma\Delta t) = \Delta t + O(\Delta t^2)$.
        \item $C A_{21} A_{22} = -C \tfrac{\lambda^2}{\gamma^2}\Delta t (1+O(\Delta t)) = - \frac{C\lambda^2}{\gamma^2}\Delta t + O(\Delta t^2)$.
        \item The $\theta$ term involves $A_{11}A_{22} + A_{12}A_{21}$. Note that $A_{11}A_{22} = (1-\tfrac{\lambda}{\gamma}\Delta t)(1-\gamma\Delta t)(1+\tfrac{\lambda}{\gamma}\Delta t) = 1 - \gamma\Delta t + O(\Delta t^2)$, and $A_{12}A_{21} = O(\Delta t^2)$.
        Thus, $\theta(A_{11}A_{22} + \dots) = \theta(1-\gamma\Delta t) + O(\Delta t^2)$.
    \end{itemize}
    Combining these:
    \[
    (A^T P A)_{12} = \Delta t - \frac{C\lambda^2}{\gamma^2}\Delta t + \theta - \gamma\theta\Delta t + O(\Delta t^2) = \theta - \Delta t \left( \gamma\theta - 1 + \frac{C\lambda^2}{\gamma^2} \right) + O(\Delta t^2).
    \]
    Since $P_{12}=\theta$, we have:
    \[
    Q_{0,12} = \lim_{\Delta t \to 0} \frac{\theta - (A^T P A)_{12}}{\Delta t} = \gamma\theta - 1 + \frac{C\lambda^2}{\gamma^2}.
    \]
    
    \paragraph{4. $Q_0$ positive-definiteness}
    To ensure $Q_0$ is positive definite ($Q_0 \succ 0$), we choose $\theta$ to eliminate the off-diagonal entries:
    \[
    \gamma\theta - 1 + \frac{C\lambda^2}{\gamma^2} = 0 \quad \implies \quad \theta = \frac{1}{\gamma} - \frac{C\lambda^2}{\gamma^3}.
    \]
    Substituting this into the diagonal entries yields:
    \begin{align*}
        Q_{11} &= \frac{2\lambda}{\gamma} + \frac{2\lambda^2}{\gamma^2}\left( \frac{1}{\gamma} - \frac{C\lambda^2}{\gamma^3} \right) = \frac{2\lambda}{\gamma} \left( 1 + \frac{\lambda}{\gamma^2} - \frac{C\lambda^3}{\gamma^4} \right), \\
        Q_{22} &= 2C\gamma - \frac{2C\lambda}{\gamma} - 2\left( \frac{1}{\gamma} - \frac{C\lambda^2}{\gamma^3} \right) = 2C \left( \gamma - \frac{\lambda}{\gamma} + \frac{\lambda^2}{\gamma^3} \right) - \frac{2}{\gamma}.
    \end{align*}
    Requiring $Q_{11} > 0$ yields the upper bound $C < \frac{\gamma^4}{\lambda^3} + \frac{\gamma^2}{\lambda^2}$. Requiring $Q_{22} > 0$ yields the lower bound $C > \frac{1}{\gamma^2 - \lambda + \lambda^2/\gamma^2}$. Combined with $C > \theta^2$ for $P \succ 0$, these conditions define the sufficient parameter region presented in \eqref{eq:continuous-conditions}.

    \paragraph{5. Nonemptiness of the parameter region.}
It remains to check that the parameter region defined by
\eqref{eq:continuous-conditions}, together with the constraint $P\succ0$,
is nonempty for every $\lambda,\gamma>0$. Set
\[
    r := \frac{\lambda}{\gamma^2} > 0.
\]
Then the lower and upper bounds in \eqref{eq:continuous-conditions} can be written as
\[
    C_L = \frac{1}{\gamma^2(r^2-r+1)},
    \qquad
    C_U = \frac{1+r}{\gamma^2 r^3}.
\]
Since
\[
    (1+r)(r^2-r+1)=r^3+1>r^3,
\]
we have $C_L<C_U$ for every $r>0$, hence for every $\lambda,\gamma>0$. It remains to verify the positivity of $P$ on this interval. Since
\[
    P=\begin{pmatrix}1&\theta\\ \theta&C\end{pmatrix},
\]
the condition $P\succ0$ is equivalent to $C>\theta^2$. With the substitution
\[
    s:=C\gamma^2,
    \qquad
    \theta=\frac{1-sr^2}{\gamma},
\]
this condition becomes
\[
    s>(1-sr^2)^2.
\]
Define $f(s):=s-(1-sr^2)^2
    =
    -r^4s^2+(1+2r^2)s-1\,.$
Since $f''(s)=-2r^4<0$, the function $f$ is strictly concave. At the two endpoints
\[
    s_L=\frac{1}{r^2-r+1},
    \qquad
    s_U=\frac{1+r}{r^3},
\]
a direct calculation gives
\[
    f(s_L)=\frac{r}{(r^2-r+1)^2}>0,
    \qquad
    f(s_U)=\frac{1}{r^3}>0.
\]
By concavity, $f(s)>0$ for every $s\in[s_L,s_U]$. Therefore, for every
$C\in(C_L,C_U)$ with
\[
    \theta=\frac{1}{\gamma}-\frac{C\lambda^2}{\gamma^3},
\]
the matrix $P$ is positive definite. For such a choice of $C$ and $\theta$, the off-diagonal entries of $Q_0$ vanish by construction, while the bounds $C_L<C<C_U$ give $Q_{0,22}>0$ and $Q_{0,11}>0$, respectively. Hence $Q_0\succ0$ and $P\succ0$ simultaneously. This proves the nonemptiness of the admissible parameter region.
\end{proof}

\section{Parameter values used in numerical experiments}
\label{app:parameters}
 
\Cref{tab:parameters} lists the parameter values used in the numerical
experiments of \cref{sec: numerics} and \cref{fig:mass-concentration}.
The same Stein metric $P$ (depending only on $\lambda$ and $\gamma$ at
$\Delta t = 1$) is used within each column. All simulations use random
seed $42$ for reproducibility, time horizon $N_{\textup{steps}} = 100$,
and step size $\Delta t = 1$.
 
\begin{table}[H]
    \centering
    \small
    \begin{tabular}{l c c c c}
        \toprule
        & \multicolumn{2}{c}{\cref{fig:mass-concentration}}
        & \multicolumn{2}{c}{\cref{fig:log-energy-decay}} \\
        \cmidrule(lr){2-3} \cmidrule(lr){4-5}
        Parameter
        & $d = 1$ panels    & $d = 5$ panels
        & $d = 1$ panels    & $d = 5$ panels \\
        \midrule
        Dimension $d$            & $1$    & $5$     & $1$    & $5$ \\
        Attraction $\lambda$     & $0.4$  & $0.8$   & $0.5$  & $0.8$ \\
        Friction $\gamma$        & $0.6$  & $0.6$   & $0.6$  & $0.6$ \\
        Noise scale $\sigma$     & $0.3$  & $0.2$   & $0.3$  & $0.2$ \\
        Noise floor $\sigma_0$   & $0.3$  & $0.05$  & $\{0.01,\,0.03\}$ & $\{0.01,\,0.03\}$ \\
        Weighting $\alpha$       & $50$   & $30$    & $50$   & $30$ \\
        Particles $N$            & $100$  & $300$   & $100$  & $300$ \\
        Measurement radius $r$   & $\{0.3,\,0.1\}$  & $\{0.3,\,0.1\}$  & ---    & --- \\
        Initial law              & $\mathcal{N}(4,\,1)$
                                 & $\mathcal{N}(0,\,2^2 I_d)$
                                 & $\mathcal{N}(4,\,1)$
                                 & $\mathcal{N}(0,\,2^2 I_d)$ \\
        \bottomrule
    \end{tabular}
    \caption{Parameter values used in the numerical experiments. Values in
    braces (e.g.\ ``$\{0.01,\,0.03\}$'' for $\sigma_0$ in the energy decay
    figure) indicate that the corresponding figure compares two trajectories
    with the same setup except for that parameter. Common to all simulations:
    time horizon $N_{\textup{steps}} = 100$, step size $\Delta t = 1$,
    random seed $42$.}
    \label{tab:parameters}
\end{table}

\end{document}